\documentclass[a4paper,12pt]{amsart}
 \usepackage{amssymb,amsfonts,amsmath}
\usepackage{bbm}
\usepackage[top=2.5cm, bottom=2.5cm, right=2.5cm, left=2.5cm]{geometry}

\def\d{\delta}
\def\H{\mathcal{H}}

\def\C{\mathbb{C}}
\def\c2{\mathbb{C}^2}
\def\R{\mathbb{R}}
\def\Q{\mathbb{Q}}

\def\N{\mathbb{N}}

\def\P{\mathbb{P}}

\def\E{\mathcal{E}}

\def\1{\mathbf{1•}}

\def\a{\alpha}

\def\e{\varepsilon}

\def\f{\varphi}

\def\p{\psi}

\def\om{\omega}

 \newcommand{\MA}{\mathrm{MA}}
\newcommand{\Amp}{\mathrm{Amp}\,}

\newtheorem{lem}{Lemma}[section]
\newtheorem{prop}[lem]{Proposition}
\newtheorem{defi}[lem]{Definition}
\newtheorem{def/not}[lem]{Definition/Notations}

\newtheorem{thm}[lem]{Theorem}

\newtheorem{cor}[lem]{Corollary}
\newtheorem{rem}[lem]{Remark}

\newtheorem*{ackn}{Acknowledgements}

\newcommand{\Real}{\mathbb{R}}

\newcommand{\HH}{\mathcal{H}}

\newcommand{\Cinf}{C^{\infty}}

\newcommand{\set}[1]{\left\{#1\right\}}

\DeclareMathOperator{\Ric}{\mathrm{Ric}} %
\newcommand{\psh}{{\rm PSH}}

\begin{document}

\title[The  metric space of K\"{a}hler currents]{Geometry and topology of the space of K\"{a}hler metrics on singular varieties}

\author{Eleonora Di Nezza}
\email{dinezza@ihes.fr}
\address{Department of Mathematics, Imperial College London, London SW7 2AZ, UK}
\curraddr{IHES, Universit\'e Paris Saclay, 91400 Bures sur Yvette, France}
\author{Vincent Guedj}
\email{vincent.guedj@math.univ-toulouse.fr}
\address{Institut de Math\'ematiques de Toulouse
Universit\'e de Toulouse ; CNRS
UPS, F-31062 Toulouse Cedex 9, France}

\thanks{ The first author is supported by a Marie Curie fellowship 660940- KRF-CY. The second author is partially supported by the ANR project GRACK}

\begin{abstract}
Let $Y$ be a compact K\"ahler normal space and
$\a \in H^{1,1}_{BC}(Y)$ a K\"ahler class.
We study metric properties of the space 
$\HH_\a$ of K\"ahler metrics in $\a$ using Mabuchi geodesics.
We extend several results by Calabi, Chen and Darvas 
previously established when the underlying space is smooth.
As an application we  analytically characterize the existence of K\"ahler-Einstein
metrics on $\Q$-Fano varieties, generalizing a result of Tian,
and illustrate these concepts in the case of toric varieties.
\end{abstract}

\maketitle

%\vspace*{6pt}\tableofcontents  % for this guide only.
% A table of contents should normally not be included

\section*{Introduction}
\label{sec:intro}
 
 Let $Y$ be a compact K\"ahler manifold and $\a_Y \in H^{1,1}(Y, \R)$ a K\"ahler class. The space $\HH_{\a_Y}$ of K\"ahler metrics $\omega_Y$ in $\a_Y$ can be seen as an infinite dimensional riemannian manifold
whose tangent spaces $T_{\omega_Y} \H_{\a_Y}$ can all be identified with ${\mathcal C}^{\infty}(Y,\R)$.
Mabuchi has introduced in \cite{Mab87} an $L^2$-metric on $\H_{\a_Y}$, by setting
$$
\langle f, g \rangle_{\omega_Y}:=\int_Y f \, g \, \frac{{\omega_Y}^n}{V_{\a_Y}},
$$
where $n=\dim_\C Y$ and $V_{\a_Y}=\int_Y {\omega_Y}^n=\a_Y^n$ denotes the volume of $\a_Y$.

Mabuchi  studied the corresponding geometry of $\H_{\a_Y}$, showing in particular that it
can formally be seen as a locally symmetric space of non positive curvature. Semmes \cite{Sem92} re-interpreted the geodesic equation as a complex homogeneous equation, while Donaldson \cite{Don99}
strongly motivated the search for smooth geodesics through its connection with the uniqueness of constant scalar curvature K\"ahler metrics.

In a series of remarkable works \cite{Chen00,CC02,CT08,Chen09,CS09} X.X.Chen
and his collaborators have studied the metric and geometric pro\-perties of the space $\H_{\a_Y}$ showing in particular that it is a path metric space (a non trivial assertion in this infinite dimensional setting).
A key step from \cite{Chen00} has been to produce 
${\mathcal C}^{1,\overline{1}}$-geodesics which turn out to minimize the intrinsic  distance $d$.  Very recently such a regularity result was improved by Chu-Tosatti-Weinkove \cite{CTW17}: they showed that geodesics are $C^{1,1}$. It follows from the work of Lempert-Vivas \cite{LV11}, Darvas-Lempert \cite{DL12} and Ross-Witt-Nystr\"om \cite{RWN15}
 that one can not expect better regularity, but for the toric setting.

The metric study of the space $(\H_{\a_Y},d)$ has been recently pushed further by Darvas in 
%a series of works
 \cite{Dar13,Dar14,Dar15}.
He characterized there the metric completion of $(\H_{\a_Y},d)$ and showed that such a completion is non-positively curved in the sense of Alexandrov. He also introduced several Finsler type metrics on $\H_{\a_Y}$, 
which turn out to be quite useful (see  \cite{DR15,BBJ15}).
For each $p \geq 1$, we set
\begin{equation}\label{distance smooth}
d_p(\phi_0,\phi_1):=\inf \{ \ell_p(\phi) \, | \, \phi
\text{ is a path  joining } \phi_0 \text{ to } \phi_1 \},\; \; \forall \phi_0, \phi_1\in \H_{\omega_Y}
\end{equation}
where 
$$
\ell_p(\phi):=\int_0^1 |\dot{\phi_t}|_p dt=\int_0^1 \left( \int_Y \left| \dot{\phi}_t \right|^p MA({\phi_t}) \right)^{1/p} dt,
$$
and $MA({\phi_t}):= (\omega_Y+dd^c \phi_t)^n /{V_{\alpha_Y}}.$ The goal of this article is to extend these studies to the case when the underlying space has singularities. 

From now on, let $Y$ be a compact K\"ahler normal space and $\a_Y \in H^{1,1}_{BC}(Y)$ a K\"ahler class, where $H^{1,1}_{BC}(Y)$ denotes the Bott-Chern cohomology space.
We fix a base point $\omega_Y$ representing $\a_Y$ and work with the space of
K\"ahler potentials $$\H_{\omega_Y}:=\left\{ \phi \in C^\infty (Y, \R)\, : \, \omega_Y+dd^c \phi \; {\textrm{is a K\"ahler form}} \right\}.$$
Our first main result extends  the main results of \cite{Chen00} and \cite[Theorem 1]{Dar15} 
as follows:

 \medskip

\noindent {\bf Theorem A.} 
{\it 
\text{ }
\begin{itemize}
\item $(\HH_{\omega_Y},d_p)$ is a metric space;
%\item the Mabuchi geodesics $(\phi_t)$ are metric geodesics;
%\item any sequence of asymptotically length minimizing paths converge, in the Hausdorff topology, to the unique geodesic.
\item 
$
d_p(\phi_0,\phi_1)=\left(\int_Y |\dot{\phi}_0|^p MA({\phi_0}) \right)^{1/p}=\left(\int_Y |\dot{\phi}_1|^p MA({\phi_1}) \right)^{1/p}, \quad \;\forall \phi_0, \phi_1\in \H_{\omega_Y}.
$
%\\ where $MA(\phi_i)$, $i=0,1$, has to be thought as the Monge-Amp\`ere measure of the smooth potential $\phi_i$ on the regular part $Y^{reg}$ of $Y$.
\end{itemize}
}
 \medskip
 As we are going to discuss in Remark \ref{rem:lenght_semipositive}, the singularities of $Y$ prevent us from defining the distance $d_p$ as in \eqref{distance smooth}. We instead work on a resolution of $Y$ and there we define $d_p$ as a limit of path length metrics. We refer to Denition \ref{distance_semipos} and Remark \ref{def d singular space} for the precise definition of $d_p$.
% The proof proceeds by metric approximation by the  spaces of K\"ahler potentials $\HH_{\pi^*\omega_Y+\e \omega_X}$,
% in a desingularization $\pi: X \rightarrow Y$.
%We show that $t \mapsto \phi_t(x)$ is uniformly Lipschitz and $x \mapsto  \phi_t(x)$
%is in ${\mathcal C}^{1,\bar 1}(Y^{reg})$ so  $\int_Y | \dot{\phi}_t |^p MA({\phi_t})$
%is well defined for a.e. $t$.

 \medskip
 Following \cite{Dar14,Dar15} we then study the metric completion  of the space $(\H_{\a_Y},d_p)$
  and establish the following generalization of \cite[Theorem 2]{Dar15}:
  
 \medskip

\noindent {\bf Theorem B.}
{\it Let $Y$ be a projective normal variety and assume $\omega_Y$ is a Hodge form.
The metric completion of $(\H_{\omega_Y}, d_p)$ is a geodesic metric space which 
%is bi-Lipschitz equivalent to 
can be identified with the finite energy class $({\mathcal E}^p(Y,\omega_Y),I_p)$.
}
 \medskip

\noindent Finite energy classes have been introduced in \cite{GZ07} and further studied in \cite{BEGZ10,BBGZ}, we recall their definition in Section \ref{sec:energy}. The Mabuchi geodesics can be extended
 to finite energy geodesics which are still metric geodesics. A key technical tool
 here is Theorem \ref{thm:equiv} which compares $d_p$ and $I_p$, where
 $$ I_p(\phi_0,\phi_1):=\left( \int_Y |\phi_0-\phi_1|^p \left[ \frac{MA(\phi_0)+MA(\phi_1)}{2} \right] \right)^{1/p} $$
 This is a natural quantity
 which allows one to define the "strong topology" on ${\mathcal E}^p(Y,\omega_Y)$ 
 %(we refer to section \ref{subsection: E^p} for the definition of $I_p$).
 
 \smallskip
The metric completion  of $(\H_{\a_Y},d)$ has been considered by Streets in his  
study of the Calabi flow \cite{Str13} and also plays an important role in 
recent  works by 
%Berman-Berndtsson \cite{BerBer14}, 
Berman-Boucksom-Jonsson \cite{BBJ15}
and Berman-Darvas-Lu \cite{BDL16}.
There is no doubt that the extension to the singular setting will play a leading role in subsequent applications.
 We illustrate this here by generalizing Tian's analytic criterion \cite{Tian97,PSSW08},
 using  results of \cite{BBEGZ} and an idea of \cite{DR15}:
 
  \medskip

\noindent {\bf Theorem C.}
{\it
Let $(Y,D)$ be a log Fano pair. It admits a unique K\"ahler-Einstein metric iff 
there exists $\e,M>0$ such that for all $\phi \in \H_{norm}$,
$$
{\mathcal F}(\phi) \leq -\e d_1(0,\phi)+M.
$$
}
 %\smallskip
 
Here ${\mathcal F}$ is a functional whose critical points are K\"ahler-Einstein potentials (Section \ref{sec:csck})
and $\H_{norm}$ is the set of potentials in $\H_{\omega_Y}$ normalized such that the supremum is $0$. This result has been independently obtained by T.Darvas \cite{Dar16} by a  different approach.
   
Our results should also be useful in analyzing more generally cscK metrics on midly singular varieties (see e.g. the recent construction by Arezzo and Spotti of cscK metrics on crepant resolutions of Calabi-Yau varieties with non-orbifold singularities \cite{AS15}).   

A way to establish the above results is to consider a resolution of singularities 
$\pi:X \rightarrow Y$ and to work with the space $\H_{\omega}$ of potentials associated to the
form $\omega=\pi^* \omega_Y$. All the above results actually hold in the more general setting 
when $\omega$ is merely a semi-positive and big form (i.e. $\int_X \omega^n>0$).
We approximate $   \H_{\omega}$ by spaces of K\"ahler potentials 
$\H_{\omega+\e \omega_X}$ and show that the most important metric properties of
$(\H_{\omega+\e \omega_X},d_\e)$ pass to the limit.
   
   \smallskip

The organization of the paper is as follows. {\it Section \ref{sec:kahler_metrics}}
starts by a recap of  Mabuchi geodesics and metrics.
{\it Theorem A} is proved  in {\it Section \ref{sec:regularity}}, where we develop
a low-regularity approach for understanding geodesics by approximation.
We introduce in {\it Section \ref{sec:energy}} classes of finite energy currents and compare
their natural topologies with the one induced by the Mabuchi distances in {\it Section \ref{sec:comparison}}.
 We study finite energy geodesics in {\it Section \ref{sec:geod}} and prove {\it Theorem B.}
We finally prove {\it Theorem C} in {\it Section \ref{sec:csck}}.

\section{The space of K\"{a}hler currents}
\label{sec:kahler_metrics}

Let $(Y,\omega_Y)$ be a {compact} K\"ahler normal space of dimension $n$. 
It follows from the definition of $H^{1,1}_{BC}(Y)$ (see for example \cite[Definition 4.6.2]{BEG}) that any
 other K\"{a}hler metric on $Y$ {in the same Bott-Chern cohomology class} of $\omega_Y$ can be written as
\begin{equation*}
    \omega_{\phi} = \omega_Y + dd^c \phi ,
\end{equation*}
where $d=\partial+\overline{\partial}$ and $d^c=\frac{1}{2i\pi} (\partial -\overline{\partial})$.
Let $\HH_{\omega_Y}$ be the space of \emph{K\"{a}hler potentials}
\begin{equation*}
    \HH_{\omega_Y} = \set{\phi \in \Cinf(Y,\R);	\; \omega_{\phi} = \omega + dd^c  \phi >0}.
\end{equation*}
This is a convex open subset of the Fr\'{e}chet vector space $\Cinf(Y):=\Cinf(Y,\R)$, thus itself a Fr\'{e}chet manifold, which is moreover parallelizable : 
$$
T\HH_{\omega_Y} = \HH_{\omega_Y} \times \Cinf(Y). 
$$
For any $\phi\in \HH_{\omega_Y}$, each tangent space $T_\phi\HH_{\omega_Y}$ is identified with $\Cinf(Y)$.

As two K\"{a}hler potentials define the same metric when (and only when) they differ by an additive constant,
we set
\begin{equation*}
    \HH_{\a_Y} = \HH_{\omega_Y} /\Real
\end{equation*}
where $\Real$ acts on $\HH_{\omega_Y}$ by addition. The set $\HH_{\a_Y}$ is therefore the {space of K\"{a}hler metrics on $Y$ in 
the cohomology class $\a_Y:=\{\omega_Y\} \in H^{1,1}_{BC}(Y)$}.

 In the whole article we fix $\pi:X \rightarrow Y$ a resolution of singularities and set $\omega=\pi^*\omega_Y$,
 $\a=\pi^*\a_Y$. Since $\a$ is no longer K\"ahler, we fix $\omega_X$ a K\"ahler form on $X$ and set
 $$
 \omega_\e:=\omega+\e \omega_X,
 $$
for $\e>0$. We will study the geometry and the topology of the spaces 
$$
\HH_{\a}=\pi^*  \HH_{\a_Y}
\; \text{ and } \; 
\HH_{\omega}=\pi^*\HH_{\omega_Y}
$$
by approximating them by the spaces $\HH_{\a_\e}, \HH_{\omega_\e}$, where
$$
\HH_{\omega_\e}:=\set{\f \in \Cinf(X,\R) \: ; \omega_\e + dd^c  \f >0}
\; \text{ and } \; 
\alpha_\varepsilon:=\{\omega_\varepsilon\}.
$$

All the properties that we are going to establish
actually hold for cohomology classes $\a$ that are merely {\it semi-positive} and {\it big} (not necessarily
the pull-back of a K\"ahler class under a desingularization).

 Our analysis will focus on the ample locus of $\a$:

\begin{defi}
The ample locus $\Amp(\a)$ of $\a$ is the Zarisiki open set of those points $x \in X$ such that
$\a$ can be represented by a positive closed $(1,1)$-current which is a smooth positive form near $x$.
\end{defi}
We then let $\mathcal{H}_\omega$ denote the space of potentials $\f\in C^\infty(X, \R)$ such that $\omega_\varphi$ is a K\"ahler form in $\Amp(\a)$. In our main case of interest, i.e. when $\a=\pi^* \a_Y$ for some K\"ahler class $\alpha_Y$ on a normal space $Y$, the ample locus 
$$
\Amp(\a)=\pi^{-1}(Y^{reg})
$$ 
is the preimage of the set of regular points of $Y$.

\subsection{The Riemannian structure}

\subsubsection{Mabuchi geodesics}\label{subsec: Mabuchi}

\begin{defi} \cite{Mab87}
The \emph{Mabuchi metric}  is the $L^{2}$ Riemannian metric on $\HH_\omega$. It is defined by
\begin{equation*}
    <\psi_{1},\psi_{2}>_{\f} = \int_{X} \psi_{1}\psi_{2}\, \frac{(\omega+dd^c \f)^n}{V_\a}
\end{equation*}
where $\f \in \HH_\omega$, $\psi_{1},\psi_{2} \in \Cinf(X)$ 
and ${(\omega+dd^c \f)^n}/{V_\a}$ is the volume element, normalized so that
it is a probability measure. Here
$
V_\a:=\a^n=\int_X \omega^n.
$
\end{defi}

In the sequel we shall also use the notation $\omega_\f:=\omega+dd^c \f$ and 
$$
MA(\f):=V_\a^{-1} {\omega_\f^n}.
$$
Geodesics between two points $\f_{0}$, $\f_{1}$ in $\HH_\omega$ correspond to the extremals of the
{Energy functional}
\begin{equation*}
 \f \mapsto  H(\f)=\frac{1}{2} \int_{0}^{1}\int_{X}(\dot{\f_{t}})^{2}\, MA({\f_{t}}) \, dt.
\end{equation*}
where $\f = \f_{t}$ is a smooth path in $\HH_\omega$ joining $\f_{0}$ and $\f_{1}$. 
The geodesic equation is formally obtained by computing the Euler-Lagrange equation for this Energy functional (with fixed end points).
It is given by 
\begin{equation}
\label{equ:geodesic_equation}
   \ddot{\f} \, MA(\f)=\frac{n}{V_\a} d \dot{\f} \wedge d^c \dot{\f}  \wedge \omega_\f^{n-1}.
\end{equation}

We are interested in the  {boundary value problem} for the geodesic equation: 
given $\f_{0},\f_{1}$ two distinct points in $\HH_\omega$, 
can one find a path $(\f(t))_{0 \leq t \leq 1}$  in $\HH_\omega$ which is a solution of  ~\eqref{equ:geodesic_equation} with 
end points $\f(0) = \f_{0}$ and $\f(1) = \f_{1}$ ?

%It has been observed by Semmes \cite{Sem92} that this can be re-formulated as a Dirichlet problem
%for the  {homogeneous complex Monge-Amp\`{e}re equation}. 
For each path $(\f_{t})_{t \in [0,1]}$ in $\HH_\omega$, we set
\begin{equation*}
    \f\left(x, t+is \right) = \f_{t}(x), \qquad x \in X, \quad t+is \in S=\{ z\in \C \;: \;  0< \Re(z)< 1\};
\end{equation*}
%\begin{equation*}
%    \f\left(x,e^{t+is} \right) = \f_{t}(x), \qquad x \in X, \quad e^{t+is} \in A = [1,e] \times \Circle;
%\end{equation*}
i.e. we associate to each path $(\f_{t})$ a 
function  $\f$ on the complex manifold $M=X\times S$, which only depends  on the real part of the strip coordinate:
we consider  $S$ as a Riemann surface with boundary 
and use the complex coordinate $z = t+is$ to parametrize the strip $S$. Set $\omega(x,z):=\omega(x)$.

Semmes observed in \cite{Sem92} 
that the path $\f_{t}$ is a geodesic in $\HH_\omega$ if and only if the associated function $\f$ on $X\times S$ is 
a $\omega$-psh solution of the homogeneous complex Monge-Amp\`{e}re equation 
\begin{equation} \label{Semmes}
(\omega+dd^c_{x,z} \f)^{n+1} = 0.
\end{equation}

This motivates the following:

\begin{defi}\label{defi: geodesic}
The function
$$
\f =\sup \{ u \, ; \,  u \in PSH(M,\omega) \text{ and } u \leq \f_{0,1} \text{ on } \partial M\}
$$
is the Mabuchi geodesic joining $\f_0$ to $\f_1$.
\end{defi}
 
Here $PSH(M,\omega)$ denotes the set of $\omega$-psh functions on $M$: these are functions $u: M\rightarrow \R\cap \{-\infty\}$ which are locally the sum of a plurisubharmonic and a smooth function and such that $\omega+dd_{x,z}^c u \geq 0 $ in the sense of currents (see section \ref{qpsh} for more details).

\begin{prop} \label{prop:Semmes}
Let $(\f_t)_{0 \leq t \leq 1}$ be the Mabuchi geodesic joining $\f_0$ to $\f_1$.
Then

(i) $\f \in PSH(M,\omega)$ is uniformly bounded on $M$ and continuous on $\Amp(\{\omega\}) \times \bar{S}$.

\smallskip

(ii) $|\f(x,z)-\f(x,z') | \leq A |\Re(z)-\Re(z') |$ with $A=\| \f_0-\f_1 \|_{L^{\infty}(X)}$.

\smallskip

(iii) $\f_{| \{\Re (z)=0 \}}=\f_0$, $\f_{| \{\Re (z)=1 \}}=\f_1$ and $(\omega+dd^c_{x,z} \f)^{n+1} = 0$.

\smallskip

\noindent It is moreover the unique bounded $\omega$-psh solution to this Dirichlet problem.
\end{prop}

We thank Hoang Chinh Lu for sharing his ideas on the continuity of $\f$.

\begin{proof}
 The proof follows from a classical balayage technique, together with a barrier argument as
 noted by  Berndtsson \cite{Bern13}. 
 Set $A=\| \f_1-\f_0\|_{L^{\infty}(X)}$. 
 
 Observe that the function $\f_0-At$, with $t=\Re(z)$, is $\omega$-psh on $M$ and 
 $\f_0-At|_{\partial M}\leq \f_{0,1}$. Hence it belongs to the family ${\mathcal F}$
 defining the upper envelope $\f$, so $\f_0-At \leq \f_t $. 

Similarly $\f_0+At$ is a $\omega$-psh function on $M$ and $\f_0+At|_{\partial M} \geq \f_{0,1}$. 
Since $(\omega+dd^c_{x,z} (\f_0+At))^{n+1}=0$, it follows from the maximum principle that $u \leq \f_0+At$,
 for any  $u \in {\mathcal F}$ in the family. Therefore 
 $$ 
 \f_0-At\leq \f_t \leq \f_0+At.
 $$
 
  Similar arguments show that 
 $$
 \f_1+A(t-1)\leq \f_t\leq \f_1-A(t-1).
 $$
 
 The upper semi-continuous regularization $\f^*$ of $\f$   satisfies the same estimates, showing in particular
that $\f^*|_{\partial M}= \f_{0,1}$. Since $\f^*$ is $\omega$-psh, we infer $\f^* \in {\mathcal F}$ hence $\f^*= \f$. 
Thus  $\f$ is $\omega$-psh and uniformly bounded, proving the first statement in $(i)$. 
Classical balayage arguments show that  $(\omega+dd^c_{x,z} \f)^{n+1} = 0$, proving $(iii)$.
 
We now prove prove $(ii)$. Consider the function  
$$
\chi_t(x)=\max\{ \f_0(x)-A \log |z|, \f_1(x)+A (\log|z|-1) \}
$$
and note that it belongs to ${\mathcal F}$ and has the right boundary values. 

Since $\chi_-=\f_0(x)-At \leq \f$ with equality at $t=0$, we infer for all $x$,
$$
-A =\frac{\partial \chi_-}{\partial t}_{|t=0} \leq \dot{\f}_0(x).
$$
Similarly
$\chi_+=\f_1(x)+A(t-1) \leq \f$ with equality at $t=1$ yields for all $x$,
$
\dot{\f}_1(x) \leq +A =\frac{\partial \chi_+}{\partial t}_{|t=1} .
$
Since $t \mapsto \f_t(x)$ is convex
(by subharmonicity in $z$), we infer that for a.e. $t,x$,
$
-A \leq \dot{\f}_0(x) \leq \dot{\f}_t(x) \leq \dot{\f}_1(x) \leq +A.
$

\smallskip

It remains to show that $\f$ is continuous on $\Amp(\{\omega\}) \times \bar{S}$. 
We can assume without loss of generality that $\f_0 < \f_1$. 
Indeed, given any $\f_0,\f_1\in \mathcal{H}_\omega$, there exists $C>0$ such that $\f_0 <\f_1+C$. 
By Lemma \ref{lem:translate}, the Mabuchi geodesic joining $\f_0$ and $\f_1+C$ is $\psi_t=\f_t+Ct$, $t\in [0,1]$. The continuity of $(x,t)\rightarrow \psi_t(x)$ will then imply the continuity of $(x,t)\rightarrow \f_t(x)$.

We change notations slighlty, replacing the strip $S$ by  the annulus   $D:=\{z=e^{t+is}\in \C\;:\; 1\leq |w|\leq e \}$.
We are going to express the function $\f$ as a global $\Theta$-psh envelope on the compact manifold
$ X \times \P^1$, where we view the annulus $D$ as a subset of the Rieman sphere,
$\C \subset \P^1=\C \cup \{ \infty\}$. The form
$\Theta(x,z)=\omega(x)+A \omega_{FS}(z)$ is a  semi-positive and big form on  the compact K\"ahler manifold
$\widetilde{M}:=X\times \P^1$, so the viscosity approach of \cite{EGZ16} can be applied
showing that the envelope $\f$ is continuous on $\Amp(\{\omega\}) \times \bar{S}$. 
Here $\omega_{FS}$ denotes  the Fubini-Study metric on 
%the Rieman sphere 
$\P^1$ and $A>0$ is a constant to be chosen below. 

Consider $U = \max(U_0, U_1)$, where
$U_0(x,z):= \varphi_0(x)$ and 
$$
U_1(x,z):=\varphi_1(x)+ A(\log|z|^2-\log(|z|^2+1) +\log (e^2+1) -2).
$$
We choose $A > 0$ so large that $U(x, 1) \equiv  \f_0(x)$. Note that
$U(x,e) \equiv \f_1(x)$ since $\f_0<\f_1$. Both $U_0$ and $U_1$ are 
$\Theta$-psh on $\widetilde{M}$, hence so is $U$.

Fix $\rho$ a local potential of $A\omega_{FS}$ in $D$ such that $\rho|_{\partial D}=0$ and
let $F$ be a continuous $S^1$-invariant function on $\widetilde{M}$  such that
\begin{itemize}
\item[(a)] $F = \f_{0,1}$ on $X \times \partial D$,
\item[(b)] $F (x, z) \geq U(x,z) \geq \f_0(x)$,
\item[(c)] $F(x,z)+\rho(z) >\f_{t}(x)$ in $X\times D$, with $t=\log|z|$.
\end{itemize}
We let the reader check that the function $F=U$ in $\widetilde{M} \setminus X \times D$ and 
$$
F(x,z) := (1-\log|z|)\f_0(x)+(\log|z|)\f_1(x)-\rho(z)+(\log|z|)(1-\log|z|),
$$
for $(x,z) \in X \times D$,
does the job.

We claim that  for all $(x,z) \in X \times D$,
$$
P_\Theta (F)(x, z) + \rho(z) = \f_{\log|z|} (x)
$$
where
$$
P_\Theta (F):=\sup\{v\; : \;v\in \psh (\widetilde{M},\Theta)
\text{ and } 
 v\leq F\}.
$$
Indeed $P_\Theta (F)+\rho$ is $\omega$-psh in $ X \times D$ and has boundary values $\leq \f_{0,1}$.
 It follows from definition of the geodesic that $P_\Theta(F)+\rho \leq \f_{t}$. On the other hand, $F +\rho \geq U +\rho \in \psh(X \times D,  \omega)$ 
 and $U=\f_{0,1}$ on $\partial M$ thus $P_\Theta(F)+\rho =\f_{0,1}$ on $\partial M$. 
 Condition (c) insures that  $M=X \times D$ does not meet the contact set $\{P_\Theta (F) = F \}$
since $F+\rho>\f_t\geq P_\Theta(F)+\rho$.
It thus follows from a balayage argument \cite{BT82} that $(\Theta+dd^c P_\Theta(F))^{n+1}=0$ in $M$, and 
the maximum principle yields
\begin{equation*}
P_\Theta(F)+\rho=\f_{t}.
\end{equation*}
The continuity of $\f$ on $\Amp(\{\omega\}) \times \bar{S}$ now follows from \cite{EGZ16} 
together with the following easy observation: the arguments in \cite[Section 2.2]{EGZ16} ensure that if $F$ is a smooth function on $\widetilde{M}$, then $P_\Theta(F)$ is a $\Theta$-psh function, continuous on $\Amp(\{\Theta\})$. The same result holds if $F$ is merely continuous. Indeed, let $F_j$ be a sequence of smooth functions on $\widetilde{M}$ converging uniformly to $F$. Taking the envelope at both sides of the inequality $F_j\leq F+\|F_j-F\|_{L^\infty(X)} $ we get $P_\Theta(F_j)\leq P_\Theta(F)+ \|F_j-F\|_{L^\infty(X)} $. Hence, $\|P_\Theta(F_j)-P_\Theta(F)\|_{L^\infty(X)}\leq \|F_j-F\|_{L^\infty(X)}$. Thus $P_\Theta(F_j)$ converges uniformly to $P_\Theta(F)$, and so $P_\Theta(F)$ is a $\Theta$-psh function that is continuous on $\Amp(\{\Theta\})=\Amp(\{\omega\})\times \bar{S} $. 
\end{proof}

\begin{rem}
If one could choose $F$ smooth in the proof above, it would follow from \cite{BD12}  (or from \cite[Theorem 1.2]{Ber}) that
$\f \in {\mathcal C}^{1,\bar 1}( \Amp(\a) \times S)$.
This would also provide a compact proof of Chen's regularity result.
\end{rem}

 We now observe that  geodesics in ${\mathcal H}_{\omega}$ are projection
of those in ${\mathcal H}_{\omega_\e}$ :

\begin{prop} \label{cor:approx}
Let $\f$ denote the geodesic joining $\f_0$ to $\f_1$ in ${\mathcal H}_{\omega}$
and let $\f^{\e}$ denote the corresponding geodesic in the space ${\mathcal H}_{\omega_\e}$. 
The map $\e \mapsto \f^\e$ is increasing and $\f^\e$ decreases to $\f$ as $\e$ decreases to zero.
Moreover
$$
\f=P(\f^\e),
$$
where $P$ denotes the projection operator onto the space $PSH(M,\omega)$.
\end{prop}

Recall that, for an upper semi-continuous function $u:M \rightarrow \R$, its projection $P(u)$ is defined
by
$$
P(u):=\sup \{ v \in PSH(M,\omega) \, ; \, v \leq u \}.
$$
The function $P(u)$ is either identically $-\infty$ or belongs to $PSH(M,\omega)$.
It is the greatest $\omega$-psh function on $M$ that lies below $u$.

\begin{proof}
Set $\p:=P(\f^{\e})$. 
Since $\omega \leq \omega_\e$, it follows from the envelope point of view that $\f \leq \f^{\e}$.
Thus $\f =P(\f) \leq P(\f^{\e})=\p$ and $\p \in PSH(M,\omega)$.
Now $\p \leq \f$ since $\p \leq \f^{\e}=\f_0,\f_1$ on $\partial M$ and $\p \in PSH(M,\omega)$.
Thus $\p=P(\f^{\e})=\f$.

Fix $\e' \leq \e$. The inclusion $PSH(M,\omega_{\e'}) \subset PSH(M,\omega_\e)$ implies similarly that
$\f \leq \f^{\e'} \leq \f^{\e}$. The decreasing limit $v$ of $\f^\e$, as $\e$ decreases to zero, 
satsifies both $\f \leq v$ and $v \in PSH(M,\omega)$ with boundary values $\f_0,\f_1$, thus
$v=\f$.
\end{proof}

It will  also be interesting to consider {\it subgeodesics}: 

\begin{defi} \label{defi:subgeodesic}
A subgeodesic is a path $(\f_t)$ of functions
in $\H_\omega$ (or in larger  classes of $\omega$-psh functions) such that the associated function is a $\omega$-psh function on $X \times S$.
\end{defi}

 We shall soon need the following simple observation:

\begin{lem} \label{lem:translate}
Fix $c \in \R$, $\f,\p \in \HH_\omega$ and  let $(\f_t)_{0 \leq t \leq 1}$ denote the Mabuchi geodesic joining $\f=\f_0$ to $\f_1=\p$.
Then $\p_t(x):=\f_t(x)-ct, 0 \leq t \leq 1, x \in X,$ is the Mabuchi geodesic joining $\f$ to $\p-c$.
\end{lem}

\begin{proof}
The proof follows from Definition \ref{defi: geodesic} and the definition of envelopes since 
$\sup \{ v \, ; \,  v \in PSH(M,\omega) \text{ and } v \leq  \f, v\leq \psi-c \text{ on  } \partial M\} = \f_t -ct . $
% Observe that $MA(\p_t)=MA(\f_t)$ and $\ddot{\p_t}=\ddot{\f_t}$, while
%$\dot{\p_t}=\dot{\f_t}-c$ hence $d \dot{\p_t}= d\dot{\f_t}$ and $d^c \dot{\p_t}= d^c \dot{\f_t}$.
\end{proof}

\subsubsection{Mabuchi and other Finsler distances}

When $\omega$ is K\"ahler, the length of a smooth path $(\f_t)_{t \in [0,1]}$ in $\HH_\omega$ is defined in a standard way,
$$
\ell(\f):=\int_0^1 |\dot{\f_t}| dt=\int_0^1 \sqrt{ \int_X \dot{\f}_t^2 MA({\f_t})} dt.
$$
The distance between two points in $\HH_\omega$ is then
$$
d(\f_0,\f_1):=\inf \{ \ell(\f) \, | \, \f
\text{ is a smooth path  joining } \f_0 \text{ to } \f_1 \}.
$$

It is easy to verify that $d$ defines a semi-distance (i.e. non-negative, symmetric and satisfying the triangle inequality). 
It is however non trivial
%, in this infinite dimensional context, 
to check that $d$ is non degenerate (see \cite{MM05} for a striking example).

Observe that $d$ induces a distance on $\HH_\a$ 
(that we abusively still denote by $d$)  compatible with the riemannian splitting
$\HH_\omega=\HH_\a \times \R$, by setting
$$
d(\omega_\f,\omega_\p):=d(\f,\p)
$$
whenever the potentials $\f,\p$ of $\omega_\f,\omega_\p$ are normalized by $E(\f)=E(\p)=0$ (see Section \ref{AubinMabuchi_functional} for the definition of the functional $E$).

\smallskip

It is rather easy to check that $(\HH_\a,d)$ is not a complete metric space. 
We shall describe the metric
completion $(\overline{\HH}_\a,d)$ in Section \ref{sec:geod}.
Following Darvas \cite{Dar15} we introduce a family of distances that generalize $d$:
%the Mabuchi distance:

\begin{defi}
For $p \geq 1$ and $\omega$ K\"ahler, we set
$$
d_p(\f_0,\f_1):=\inf \{ \ell_p(\f) \, | \, \f
\text{ is a smooth path  joining } \f_0 \text{ to } \f_1 \},
$$
where 
$
\ell_p(\f):=\int_0^1 |\dot{\f_t}|_p dt=\int_0^1 \left( \int_X \left| \dot{\f}_t \right|^p MA({\f_t}) \right)^{1/p} dt.
$
\end{defi}

Note that $d_2=d$ is the Mabuchi distance.
%It is again easy to verify that $d_p$ defines a semi-distance.
%The fact that 
Mabuchi geodesics have constant speed with respect to all the Finsler structures $\ell_p$,
as was observed
by Berndtsson \cite[Lemma 2.1]{Bern09b}: for any ${\mathcal C}^1$-function $\chi$,
$$
t \mapsto \int_X \chi(\dot{\f}_t) MA(\f_t)
$$
is constant along a geodesic. Indeed 

\begin{eqnarray*}
\frac{d}{dt} \int_X \chi(\dot{\f}_t) MA(\f_t)&=& \int_X \chi'(\dot{\f}_t) \ddot{\f}_t MA(\f_t)
+\frac{n}{V_\alpha} \int_X \chi(\dot{\f}_t) dd^c \dot{\f}_t \wedge \omega_{\f_t}^{n-1} \\
&=& \int_X \chi'(\dot{\f}_t) \left\{ \ddot{\f}_t MA(\f_t) -\frac{n}{V_\alpha} d \dot{\f}_t  \wedge d^c \dot{\f}_t \wedge \omega_{\f_t}^{n-1} \right\}=0
\end{eqnarray*}

since $\ddot{\f}_t MA(\f_t) -\frac{n}{V_\alpha} d \dot{\f}_t  \wedge d^c \dot{\f}_t \wedge \omega_{\f_t}^{n-1}=0$.
Applying this observation to $\chi(t)=t^p$ shows that Mabuchi geodesics have constant $\ell_p$-speed.

\medskip

When $\omega$ is merely semi-positive there are fewer smooth paths within $\H_{\omega}$.
It is natural to consider smooth paths in $\H_{\omega_\e}$ and pass to the limit in
the previous definitions :
 
\begin{defi}\label{distance_semipos}
Assume $\omega$ is semi-positive and big.
Let $\f_0,\f_1\in \mathcal{H}_\omega$. We define the Mabuchi distance between $\f_0$ and $\f_1$ as
$$
d_p(\f_0, \f_1):= \liminf_{\varepsilon\rightarrow 0} d_{p,\varepsilon}(\f_0, \f_1),
$$
where $d_{p, \varepsilon}$ is the distance w.r.t. the K\"ahler form $\omega_\varepsilon:=\omega+\varepsilon\omega_X$.
\end{defi}

%It is again easy to check that $d_p$ is a semi-distance.
We will show in Theorem \ref{thm:approx0} that it is a distance, which moreover does not 
depend on the way we approximate $\omega$ by K\"ahler classes.

\begin{rem}\label{rem:lenght_semipositive}
For any smooth path $\psi:[0,1]\rightarrow \mathcal{H}_\omega$, we can still define
$$\ell_p(\psi):= \int_0^1 \left( \frac{1}{V}\int_X |\dot{\psi}_t|^p (\omega+dd^c \psi_t)^n  \right)^{1/p} dt$$
when $\omega$ is merely semi-positive. Since $PSH(M,\omega) \subset PSH(M,\omega_\e)$, $\psi_t$ is both in $\mathcal{H}_{\omega}$ and $\mathcal{H}_{\omega_\varepsilon}$. Observe that
\begin{eqnarray*}
V_\e^{-1} \int_X |\dot{\p}_t|^p (\omega_\e+dd^c {\p_t})^n
 & =& V_\e^{-1} \int_X |\dot{\p}_t|^p (\omega+dd^c {\p_t}+\e \omega_X)^n \\
 & \leq & V^{-1} \int_X |\dot{\p}_t|^p (\omega+dd^c {\p_t})^n +A \e,
\end{eqnarray*}
hence
$$\ell_{p,\varepsilon}(\psi) \leq \ell_p(\psi)+ A' \varepsilon$$
where $\ell_{p,\varepsilon}$ denotes the length in $\mathcal{H}_{\omega_\varepsilon}$. We infer
$$d_p(\f_0, \f_1) \leq \inf \{\ell_p(\psi) \quad \psi \;{\rm smooth\; path\; joining \; \f_0 \; and \; \f_1\; in }\; \mathcal{H}_\omega\}.$$
The converse inequality is however unclear, due to the lack of positivity of $\omega$: it is difficult to smooth out $\omega$-psh functions if $\omega$ is not K\"ahler. This partially explains Definition \ref{distance_semipos}.
\end{rem}

\subsection{Approximation by K\"ahler classes}  \label{sec:regularity}

Fix $\f_0,\f_1 \in \HH_\omega$.
We let $(\f_t)_{0 \leq t \leq 1}$ denote the Mabuchi geodesic in $\HH_{\omega}$ joining $\f_0$ to $\f_1$.
\begin{defi}
For $t=0,1$ we set
$$
I(t):=\int_X |\dot{\f}_t|^p MA(\f_t).
$$
\end{defi}

\begin{thm}  \label{thm:approx0}
Set $\omega_\e=\omega+ \e \omega_X$, $\e>0$. 
Then  $\lim_{\e \rightarrow 0} d_{p,\omega_\e}(\f_0,\f_1)$ exists and is independent of $\omega_X$. More precisely,
$$
d_{p,\e}^p (\f_0,\f_1) \rightarrow I(0)=I(1).
$$
In particular $d_{p}(\f_0,\f_1)=I(0)^{1/p}=I(1)^{1/p}$ defines a distance on $\H_{\omega}$.
\end{thm}

In the definition of $I(0), I(1)$, the time derivatives $\dot{\f}_0=\dot{\f}_0^+$, $\dot{\f}_1=\dot{\f}_1^-$ denote the right and left derivative, respectively.
\begin{rem}\label{def d singular space}
When $\omega=\pi^* \omega_Y$, for some K\"ahler form $\omega_Y$ on a compact normal space $Y$, for each $p\geq 1$ and $\forall \phi_0, \phi_1\in \H_{\omega_Y}$ we define
$$d_p(\phi_0, \phi_1):= d_p (\f_0, \f_1),\quad\;\rm{where} \;\f_0= \pi^* \phi_0, \f_1= \pi^* \phi_1.$$
This definition does not depend on the choice of resolution. Indeed let $\pi':X'\rightarrow Y$ be another resolution of $Y$ that dominates $X$, i.e. there exists an holomorphic and bimeromorphic map $f: X'\rightarrow X$ such that $\pi'= \pi\circ f$. Set $\omega':= {\pi'}^* \omega_Y= f^* \omega$. We need to show that
$$d_{p, \omega} (\f_0, \f_1)= d_{p, f^*\omega} (f^* \f_0, f^* \f_1). $$
Denote by $\psi_t$ the $ f^*\omega$-geodesic joining $f^* \f_0$ and $f^* \f_1$. We claim that $\psi_t= f^* \f_t$. We first observe that, since  $\psi_t$ is a $f^* \omega$-psh function for each fixed $t$, $\psi_t= f^* \gamma_t$ where $\gamma_t$ is a $\omega$-psh function on $X$. Set $M':= X'\times S$, $\psi(x',t):=\psi_t(x') $ and $\gamma(x',t):= \gamma_t(x')$ for each $(x',t)\in M'$. By construction we have that
$$f^*(\omega+dd^c \gamma)^{n+1}= (f^*\omega+dd^c \psi)^{n+1}=0\quad {\rm on}\;M':= X'\times S, \quad  \psi|_{\partial M'}= f^* \f_{0,1}.$$
The claim follows from the uniquennes of the solution of the Dirichlet problem above (Proposition \ref{prop:Semmes}).
The invariance of the non-pluripolar Monge-Amp\`ere measure under bimeromorphic maps \cite{DN15} together with the fact that $V:=\int_X \omega^n= \int_{X'} f^* \omega$ give
$$\int_X |\dot{\f}_0|^p \frac{(\omega+dd^c \f_0)^n}{V}= \int_{X'} |\dot{f^* \f}_0|^p \frac{(f^*\omega+dd^c f^*\f_0)^n}{ V}= \int_{X'} |\dot{\psi}_0|^p \frac{(\omega'+dd^c \f_0)^n}{V}.$$
The conclusion then follows from Theorem \ref{thm:approx0}.

\end{rem}

\begin{proof}
Observe that $\f_0, \f_1 \in \mathcal{H}_{\omega_{\varepsilon}}$ 
and let $\f_t^{\varepsilon}$ be the corresponding geodesic. 
It follows from \cite[Theorem 3.5]{Dar15} that
$$  
d_{p,\e}^p (\f_0,\f_1) =  V_{\varepsilon}^{-1}\int_X |\dot{\f}_0^\varepsilon|^p (\omega_\e +dd^c \f_0)^n.
 $$
Now  observe that 
$$
\dot{\f}_0^+\leq \dot{\f}_0^\varepsilon \leq \frac{\f_t^\varepsilon-\f_0}{t} \quad \forall t\in (0, 1)
$$
where the first inequality follows from the fact that $\e\rightarrow\f_t^\e$ is decreasing (Proposition \ref{cor:approx}),
while the second uses the convexity of $t \mapsto \f_t^\e$.
Thus
$$
 | \dot{\f}_0^\varepsilon - \dot{\f}_0^+ |\leq \left| \frac{\f_t^\varepsilon-\f_0}{t}  -  \dot{\f}_0^+ \right|.
 $$
 
Letting $\varepsilon\searrow 0$ and then $t\rightarrow 0$ shows that $| \dot{\f}_0^\varepsilon - \dot{\f}_0^+ |$ converges pointwise to zero.
Moreover, $(\omega_\varepsilon+dd^c\f_0)^n= f_\varepsilon dV$ where $dV$ is the Lebesgue measure and 
$f_\varepsilon>0$ are smooth densities which converge locally uniformly to $f\geq 0$ with 
$(\omega+dd^c \f_0)^n=f dV$. The dominated convergence theorem thus yields
$$
\lim_{\e\rightarrow 0} d_{p, \varepsilon}^p(\f_0, \f_1)=  V^{-1}\int_X |\dot{\f}_0^+|^p (\omega+dd^c \f_0)^n=I(0).
$$
 The argument for $I(1)$ is similar.
 
 This shows in particular that $d_p$ is a distance on $\H_{\omega}$:
 if $d_p(\f_0,\f_1)=0$, then $I(0)=I(1)=0$, hence
 $\dot{\f}_0(x)=\dot{\f}_1(x)=0$ for a.e. $x \in X$, which implies $\dot{\f}_t (x)=0$ for a.e. $x \in X$,
 by convexity of $t \mapsto \f_t(x)$. Thus, $\f_0(x)=\f_1(x)$ for a.e. $x \in X$.
\end{proof}

We now extend the definition of the distance $d_p$ for bounded $\omega$-psh potentials.

\begin{defi}\label{def: bd}
Let $\f_0, \f_1 \in \psh (X, \omega)\cap L^\infty(X)$ then 
$$d_p(\f_0, \f_1):= \liminf_{\varepsilon\rightarrow 0} \liminf_{j,k\rightarrow +\infty} d_{p,\e} (\f_0^j, \f_1^k) = \liminf_{\varepsilon\rightarrow 0} d_{p, \e}(\f_0, \f_1)$$
where $\f_0^j, \f_1^k$ are smooth sequences of $\omega_\e$-psh functions decreasing to $\f_0$ and $\f_1$, respectively.
\end{defi}
Observe that $d_{p, \omega_\e}(\f_0, \f_1) $ is well defined for potentials in $\mathcal{E}^p(X, \omega_\e)$ (\cite{Dar15}), and so in parti\-cular for bounded $\omega_\e$-psh functions.\\
%One can check that the definition does not depend on the choice of the approximants with the same arguments in the proof of \cite[Proposition 4.1]{DNG}. The following lemma insure that $d_p$ as in Definition \ref{def} is well defined.
\begin{prop}\label{quasi decreasing}
Let $\f_0, \f_1 \in \psh (X, \omega)\cap L^\infty(X)$. The limit of $d_{p, \omega_\e}(\f_0, \f_1)$ as $\e$ goes to zero exists and it does not depend on the choice of $\omega_X$.
\end{prop}   

\begin{proof}
First, observe that since $\f_0, \f_1$ are bounded, they belong to $\E^p(X,\omega_\e)$ for any $0\leq \e\leq 1$.
By \cite[Corollary 4.14]{Dar15} we know that the Pythagorean formula holds true, i.e.
$$d^p_{p, \e}(\f_0, \f_1) = d^p_{p,\e}(\f_0, \f_0\vee_\varepsilon \f_1)+d^p_{p,\e}( \f_0\vee_\varepsilon \f_1, \f_1),$$
where $\psi:=\f_0\vee_\varepsilon \f_1$ is the greatest $\omega_\varepsilon$-psh function that lies below $\min{(\f_0,\f_1)}$.
Fix $\e\leq \e{'}$. We claim that

$$V_\e d^p_{p, \e}(\f_0, \psi)\leq V_{\e'}d^p_{p, {\e{'}}}(\f_0, \psi) \quad {\rm and}\quad V_\e d^p_{p, \e}(\psi, \f_1)\leq V_{\e'} d^p_{p, {\e{'}}}(\psi, \f_1).$$
Let $\psi_t^\varepsilon$, $\psi_t^{\varepsilon'}$ denote the $\varepsilon$-geodesic and the $\varepsilon'$-geodesic both joining  $\psi$ and $\f_0$. Since $\varepsilon\rightarrow \psi_t^\varepsilon$ is increasing (Proposition \ref{cor:approx}) we have that for any $t\in(0,1)$
$$\frac{\psi_t^\varepsilon-\psi}{t}\leq \frac{\psi_t^{\varepsilon'}-\psi}{t}$$
that implies $\dot{\psi}^\e_0\leq \dot{\psi}^{\e'}_0$. Moreover observe that since $\f_0(x)\geq \psi(x)$ for all $x\in X$, Lemma \ref{lem:maxprince} yields $\dot{\psi}^\e_0(x)\geq 0$ for all $x\in X$. It then follows that
$$\int_X |\dot{\psi}^\e_0|^p {(\omega_\varepsilon+dd^c \psi)^n} \leq \int_X |\dot{\psi}^{\e'}_0|^p {(\omega_{\varepsilon'}+dd^c \psi)^n},$$
hence the claim.
The same type of arguments give $V_\e d^p_{p, \e}(\psi, \f_1)\leq V_{\e'}d^p_{p, {\e{'}}}(\psi, \f_1)$.
Hence $${V_\e}{V_{\e'}^{-1}}d^p_{p, \e}(\f_0, \f_1) \leq   d^p_{p, \e'}(\f_0,  \f_0\vee_\varepsilon \f_1)+ d^p_{p, {\e{'}}}( \f_0\vee_\varepsilon \f_1, \f_1).$$
Using again \cite[Corollary 4.14]{Dar15} and the triangle inequality we get
$${V_\e}{V_{\e'}^{-1}} d^p_{p, \e}(\f_0, \f_1) \leq  d^p_{p, \e'}(\f_0, \f_1) + 2 d^p_{p, \e'}( \f_0\vee_\varepsilon \f_1, \f_0 \vee_{\varepsilon'} \f_1) .$$
Moreover, since $\f_0\vee_{\varepsilon'} \f_1 \geq \f_0\vee_{\varepsilon} \f_1$, \cite[Lemma 5.1]{Dar15} yields 
\begin{eqnarray*}
d^p_{p, \e'}( \f_0\vee_\varepsilon \f_1, \f_0\vee_{\varepsilon'} \f_1) & \leq  & \frac{1}{V_{\e'}}\int_X (\f_0\vee_{\varepsilon'} \f_1-\f_0\vee_{\varepsilon} \f_1)^p\, (\omega_{\e'}+dd^c (\f_0\vee_{\varepsilon} \f_1))^n\\
&\leq & \frac{1}{V_{\e'}}\int_X (\f_0\vee_{\varepsilon'} \f_1-\f_0\vee_{\varepsilon} \f_1)^p\, (\omega+\omega_X+dd^c (\f_0\vee_{\varepsilon} \f_1))^n\\
&:= &V_{\e'}^{-1}\eta(\e, \e').
\end{eqnarray*}
Observe that $\eta(\e, \e')$ converges to $0$ as $\e'$ goes to $0$.  From above we have
$$
V_\e d^p_{p, \e}(\f_0, \f_1) \leq V_{\e'} d^p_{p, \e'}(\f_0, \f_1) + \eta (\e, \e').
$$
Hence the limit exists. 

Now, let $\omega_X, \widetilde{\omega}_X$ be two K\"ahler metrics on $X$ such that $$\omega_X \leq \widetilde{\omega}_X\leq C \omega_X $$
for some $C>0$. Assume first $\f_0\leq \f_1$. Set $\widetilde{\omega}_\e:= \omega+ \e \widetilde{\omega}_X$ and observe that $\omega_\e\leq \widetilde{\omega}_\e\leq \omega_{\e'}$ where $\e'= \e C$. Let $\f_t^\e, \tilde{\f}_t^\e$ be the geodesic w.r.t. $\omega_\e$ and $\widetilde{\omega}_\e$, respectively and observe that $\f_t^\e\leq \tilde{\f}_t^\e\leq \f_t^{\e'}$. The same arguments as above give
$$|\dot{\f}_0^\e|^p\leq |\dot{\tilde{\f}}_0^\e|^p\leq |\dot{\f}_0^{\e'}|^p$$
hence
$$ \int_X |\dot{\varphi}^\e_0|^p {(\omega_\varepsilon+dd^c \varphi_0)^n} \leq \int_X |\dot{\tilde{\varphi}}^{\e}_0|^p {(\tilde{\omega_{\varepsilon}}+dd^c \varphi_0)^n}\leq \int_X |\dot{\varphi}^{\e'}_0|^p {(\omega_{\varepsilon'}+dd^c \varphi_0)^n} .$$
The latter tells us that the limit does not depend on $\omega_X$. To get rid of the asspumption $\f_0 \leq \f_1$, one can use Pythagorean formula as above.
\end{proof}

%\subsubsection{Bounded geodesics}

An adaptation of the classical Perron envelope technique yields the following result due to Berndtsson \cite{Bern13}:

\begin{prop} \label{prop:Bo}
Assume $\f_0,\f_1$ are bounded $\omega$-psh functions. Then
$$
\f(x,z):=\sup \{ u(x,z) \, | \, u \in PSH(X \times S, \omega) \text{ with }
\lim_{t \rightarrow 0,1} u \leq \f_{0,1} \}.
$$
is the unique bounded $\omega$-psh function on $X \times S$, which is the solution of the Dirichlet problem
$\f_{| X \times \partial S}=\f_{0,1}$ with
$$
(\omega+dd^c_{x,z} \f)^{n+1}=0 
\text{ in } X \times S.
$$

Moreover $\f(x,z)=\f(x,t)$ only depends on $\Re(z)$ and 
$|\dot{\f}| \leq \| \f_1-\f_0\|_{L^{\infty}(X)}$.
\end{prop}

The proof goes exactly as that  of Proposition \ref{prop:Semmes}.
The function $\f$ (or rather the path $\f_t \subset PSH(X,\omega) \cap L^{\infty}(X)$)
is called a {\it bounded geodesic} in  \cite{Bern13}.
We use the same terminology here, as 
it turns out that bounded geodesics are geodesics in the metric sense:

\begin{prop} \label{prop:bddmin}
Bounded geodesics are metric geodesics. More precisely, if 
$\f_0,\f_1$ are bounded $\omega$-psh functions and $\f(x,z)=\f_t(x)$ is the bounded
geodesic joining $\f_0$ to $\f_1$, then for all $t,s \in [0,1]$,
$$
d_p(\f_t,\f_s)=|t-s| \, d_p(\f_0,\f_1).
$$
\end{prop}

\begin{proof} 
Let $\f_0^{j}, \f_{1}^k \in \H_{\omega_\varepsilon}$ be sequences decreasing respectively to $\f_0, \f_1$. 
It follows from the comparison principle and the uniqueness in Proposition 
\ref{prop:Bo} that $\f_{t,j}$ decreases to $\f_{t}$ as $j$ increases to $+\infty$.
From Definition \ref{def: bd}, Proposition \ref{quasi decreasing} and the fact that the identity in the statement holds in the K\"ahler setting for $d_\varepsilon$ we obtain
\begin{eqnarray*}
  d_p \left( \f_{t}, \f_{s} \right) &=&\liminf_{\varepsilon\rightarrow 0} \liminf_{j,k\rightarrow +\infty} d_{p,\varepsilon} (\f_{t,j}, \varphi_{s,k})\\
  &=& |t-s|\liminf_{\varepsilon\rightarrow 0}\liminf_{j,k\rightarrow +\infty} d_{p,\varepsilon} (\f_{0}^j, \varphi_{1}^k)= |t-s| d_p(\f_0, \f_1).
\end{eqnarray*}
\end{proof}

\begin{rem}
One can no longer expect that $d_p(\f_0,\f_1)^p = \int_X \left| \dot{\f}_{t} \right|^p MA(\f_t)$
for a.e. $t \in [0,1]$ as simple examples show. One can e.g. take $\f_0\equiv 0$ and $\f_1=\max(u,0)$, where
$u$ takes positive values, has isolated singularities and solves $MA(u)=$Dirac mass at some point: in this case
$MA(\f_1)$ is concentrated on the contact set $(u=0)$ while $\dot{\f}_{1} \equiv 0$ on this set
hence
$
\int_X \left|\dot{\f}_{1} \right|^p MA(\f_1)=0.
$
We thank T.Darvas for pointing this to us.
\end{rem}
As the above remark points out we do not have that $d_p^p(\f_0, \f_1)= I(0)=I(1)$ when $\f_0, \f_1$ are just bounded $\omega$-psh functions. Nevertheless we can still recover the formula in some special cases.

We start by recalling the following:

\begin{thm}\label{BD}
Let $f$ be a continuous function such that $dd^c f\leq C\omega_X $ on $X$, for some $C>0$. Then $P(f) $ has bounded laplacian on $\Amp(\{\omega\})$ and
\begin{equation}\label{BD_MA}
(\omega+dd^c P_{\omega}(f))^n= \mathbbm{1}_{\{P_{\omega}(f)=f\}} (\omega+dd^c f)^n.
\end{equation}
\end{thm}
The fact that $P(f)$ has locally bounded laplacian in $\Amp(\{\omega\})$ is essentially \cite[Theorem 1.2]{Ber}. We do not assume here that $f$ is smooth but one can check that the upper bound on $dd^c f$ is the only estimate needed in order to pursue Berman's approach. One can then argue as in \cite[Theorem 9.25]{GZ17} to get identity \eqref{BD_MA}.

%\begin{rem}
%For the moment being we are not sure that the proof in \cite{BD} is correct. The proof of this result follows using the $\beta$-convergence. In particular we use Chinh's arguments together with \cite{B}. Let me alos emphasise that these arguments can be extended when the class is big and nef but not in the big case. In the big case, indeed we do not have any regularity result on $P_\omega (f)$ but we are still able to infer that
%$$ (\omega+dd^c P_{\omega}(f))^n \leq 1_{\{P_{\omega}(f)=f\}} (\omega+dd^c f)^n.$$ 
%\end{rem}

\noindent Set $$\mathcal{H}_{bd}:=\{\f\in \psh(X, \omega) \cap L^\infty(X),\; \f=P_{\omega}(f)\; {\rm for\, some}\; f\in C^0(X)\, {\rm with}\, dd^c f\leq C\omega_X,\; C>0\}.$$

\begin{thm}\label{thm:magicFormulaBdd}
Assume that $\varphi_{0}, \varphi_1\in \mathcal{H}_{bd} $. Let $\f_t$ be the  Mabuchi geodesic joining $\f_0$ and $\f_1$. Then 
\begin{equation}\label{star}
{d}_p^p(\f_0, \f_1)= \int_X |\dot{\f}_0|^p MA(\f_0)= \int_X |\dot{\f}_1|^p MA(\f_1).
\end{equation}
\end{thm}

\begin{proof}

Set $\varphi_{0,\e}:=P_{\omega_{\e}}(f_0)$ and $\varphi_{1,\e}:=P_{\omega_{\e}}(f_1)$. Clearly $\varphi_{i,\e}$ decreases pointwise to $\varphi_i$, $i=1,2$. Let $\f_t^\e$ be the $\omega_\e$-geodesic joining $\varphi_{0,\e} $ and $\varphi_{1,\e}$. Combining  \cite[Theorem 3.5]{Dar15} together with  \eqref{BD_MA} we get
$$
V_\varepsilon d_{p,\e}^p(\varphi_{0,\e},\varphi_{1,\e}) = \int_X \left|\dot{\varphi}^\e_{0}\right|^p (\omega_{\e} +dd^c \varphi_{0,\e})^n=\int_{\{\varphi_{0,\e}=f_0\}} \left|\dot{\varphi}_{0}^\e\right|^p (\omega_{\e} +dd^c f_0)^n.
$$
Set $D_{\e}:= \{\varphi_{0,\e}=f_0\}$, $D_0:= \{\varphi_{0}=f_0\}$ and observe that $D_0 \subseteq D_\e$. Since $\f_{0, \e}=P_{\omega_\e}(f)$ and $\f_0= P_\omega(f)$, Theorem \ref{BD} insures that $ (\omega_{\e} +dd^c \varphi_{0,\e})^n= g_\e \omega_X^n$ and $ (\omega +dd^c \varphi_{0})^n= g_0 \omega_X^n$ where $g_\e, g_0$ are defined as
$$
g_\e:= 
\begin{cases}
0 \qquad\qquad \qquad\quad  x\notin D_\e
\\
\frac{(\omega_{\e} +dd^c f_0)^n}{\omega_X^n}  \,\quad \quad x\in D_\e
\end{cases}
\quad 
g_0:= 
\begin{cases}
0 \qquad\qquad \qquad  x\notin D_0
\\
\frac{(\omega +dd^c f_{0})^n}{\omega_X^n}  \,\quad \quad x\in D_0
\end{cases}
$$
We claim that $g_\e$ converges pointwise to $g_0$. Indeed, when $x\in D_0\subseteq D_\e$ then 
$g_\e(x)= \frac{(\omega_{\e} +dd^c f_0)^n}{\omega_X^n} (x)$ converges to $\frac{(\omega +dd^c f_0)^n}{\omega_X^n} (x)= g_0(x) $ as $\e$ goes to $0$. In the case when $x\notin D_0$, i.e. $\f_0(x)< f_0(x)$, since $\f_\e(x) $ decreases to $\f_0(x) $ as $\e$ goes to zero, we can infer that for $\e$ sufficiently small we still have $\f_\e(x)< f_0(x)$ that means $x\notin D_\e$. Hence $g_\e(x)=0=g_0(x)$. The claim is then proved.

Since $\mathbbm{1}_{D_\e}\varphi^\e_{0}= f_0=\mathbbm{1}_{D_0}\varphi_{0} $, the same arguments in Theorem \ref{thm:approx0} show that $\left|\mathbbm{1}_{D_\e}\dot{\varphi}^\e_{0}- \mathbbm{1}_{D_0}\dot{\varphi}_{0}\right|$ converges pointwise to $0$ as $\e$ goes to zero.\\
We thus infer that $\mathbbm{1}_{D_\e} |\dot{\f_0^\e}|^p g_\e$ converges pointwise to $\mathbbm{1}_{D_0} |\dot{\f_0}|^p g_0$ as $\e\rightarrow 0$. The dominated convergence theorem yields 
$$\lim_{\e\rightarrow 0} d_{p,\e}^p (\f_0, \f_1)= \lim_{\e\rightarrow 0} \int_X  \mathbbm{1}_{D_\e}\left|\dot{\varphi}^\e_{0}\right|^p (\omega_{\e} +dd^c \varphi_{0,\e})^n= \int_X \mathbbm{1}_{D_0}\left|\dot{\varphi}_{0}\right|^p (\omega +dd^c \varphi_{0})^n,$$
hence the conclusion.
\end{proof}

Observe that if $\f_0, \f_1\in \mathcal{H}_\omega$, then $\f_0\vee \f_1\in \mathcal{H}_{bd}$. Indeed since $\f_0,\f_1$ are smooth, the functions $-\f_0,-\f_1$ are quasi-plurisubharmonic, i.e. there exists $C>0$ such that $dd^c (-\f_i) \geq -C\omega_X$ for any $i=1,2$. Thus $\min(\f_0, \f_1)=-\max(-\f_0, -\f_1)$ is such that
$$dd^c \min(\f_0, \f_1)=-dd^c \max(-\f_0, -\f_1)\leq C \omega_X.$$
In particular the equality \eqref{star} holds for $d_p(\f_0,\f_0\vee \f_1 )$ and $d_p(\f_1,\f_0\vee \f_1 )$.

\section{Finite energy classes} \label{sec:energy}

We define in this section the
set ${\mathcal E}(\a)$ (resp. ${\mathcal E}^p(\a)$)
of positive closed currents $T=\omega+dd^c \f$ with full Monge-Amp\`ere mass 
(resp. finite weighted energy) in $\a$,
by defining the corresponding class ${\mathcal E}(X,\omega)$ 
(resp. ${\mathcal E}^p(X,\omega)$ ) of finite energy potentials $\f$.

\subsection{The space ${\mathcal E}(\a)$}

\subsubsection{Quasi-plurisubharmonic functions}\label{qpsh}
Recall that a function is quasi-plurisub\-harmonic if it is locally given as the sum of  a smooth and a psh function. 
In particular quasi-psh (\emph{qpsh} for short) functions are upper semi-continuous and integrable.
%They are actually in $L^p$ for all $p \geq 1$, and the induced topologies are all equivalent.
%A much stronger integrability property actually holds:
%Skoda's integrability theorem \cite{Sko} asserts indeed that $e^{-\e \f} \in L^1(X)$ if 
%$0 < \e$ is smaller than $2/\nu(\f)$,  where $\nu(\f)$ denotes the maximal logarithmic singularity (Lelong number) of $\f$ on $X$.
%
%Quasi-plurisubharmonic functions have gradient in $L^r$ for all $r<2$, but not in $L^2$ as shown by the local model $\log |z_1|$.

\begin{defi}
We let $PSH(X,\omega)$ denote the set of all $\omega$-plurisubharmonic functions. 
These are quasi-psh functions
$\f:X \rightarrow \R \cup \{-\infty\}$ such that
$$
\omega+dd^c \f \geq 0
$$
in the weak sense of currents. 
\end{defi}

The set $PSH(X,\omega)$ is a closed subset of $L^1(X)$, 
%when endowed with 
for the $L^1$-topology. 

\pagebreak[3]

%\subsubsection{Bedford-Taylor theory}
%
%Bedford and Taylor have observed in \cite{BT82} that one can define the complex Monge-Amp\`ere operator
%$$
%MA(\f):={V_\a}^{-1} (\omega+dd^c \f)^n
%$$
%for all {\it bounded} $\omega$-psh functions.
%: they showed that whenever $(\f_j)$ is a sequence of smooth
%$\omega$-psh functions locally decreasing to $\f$, then the smooth proba\-bility measures $MA(\f_j)$
%converge, in the weak sense of Radon measures, towards a unique 
%probability measure that we denote by $MA(\f)$.
%
%At the heart of Bedford-Taylor's theory lies the following {\it maximum principle}: if $u,v$ are bounded
%$\omega$-plurisubharmonic functions, then
%$$
%\hskip-3cm (MP) \hskip2cm 1_{\{v<u\}} MA(\max(u,v)) =1_{\{v<u\}} MA(u).
%$$
%This equality is elementary when $u$ is {\it continuous}, as the set $\{v<u\}$ 
%is open. When $u$ is merely {\it bounded}, this set is only open in the plurifine topology. Since Monge-Amp\`ere measures of bounded qpsh functions do not charge pluripolar sets 
%(by the Chern-Levine-Nirenberg inequalities), and since $u$ is nevertheless {\it quasi-continuous},
%this gives a heuristic justification for $(MP)$.

%The reader will easily verify that the maximum principle $(MP)$ implies the so called {\it comparison principle}:

%\begin{prop}
%Let $u,v$ be bounded $\omega$-plurisubharmonic functions. Then
%$$
%\int_{\{v<u\}}  MA(u) \leq \int_{\{v<u\}}  MA(v).
%$$
%\end{prop}

\subsubsection{The class ${\mathcal E}(X,\omega)$}

Given   $\f \in PSH(X,\omega)$, we consider  
$$
\f_j:=\max(\f, -j) \in PSH(X,\omega) \cap L^{\infty}(X).
$$
It follows from the Bedford-Taylor theory \cite{BT82} that the $MA(\f_j)$'s are well defined probability measures. Moreover, the sequence
$
\mu_j:={\bf 1}_{\{ \f>-j\}} MA(\f_j)
$
is increasing \cite[p.445]{GZ07}. Since the $\mu_j$'s all have total mass bounded from above by $1$,
%(the total mass of the measure $MA(\f_j)$),  we can consider
we consider
$$
\mu_{\f}:=\lim_{j \rightarrow +\infty} \mu_j,
$$
which is a positive Borel measure on $X$, with total mass $\leq 1$.

% Since the $\f_j$'s are decreasing, it is natural to expect that these measures converge. 
%The following monotonicity property holds:
%  
%\begin{lem}
%The sequence
%$
%\mu_j:={\bf 1}_{\{ \f>-j\}} MA(\f_j)
%$
%is increasing.
%%an increasing sequence of Borel measures.
%\end{lem}
%
%\noindent The proof is an elementary consequence of $(MP)$ (see \cite[p.445]{GZ07}).

%\begin{rem}
%Note : $t \mapsto \max(\f(x),-t)$ is a subgeodesic (Definition \ref{defi:subgeodesic}).
%\end{rem}
%Since the $\mu_j$'s all have total mass bounded from above by $1$,
%%(the total mass of the measure $MA(\f_j)$),  we can consider
%we consider
%$$
%\mu_{\f}:=\lim_{j \rightarrow +\infty} \mu_j,
%$$
%which is a positive Borel measure on $X$, with total mass $\leq 1$.

\begin{defi}
We set 
$$
{\mathcal E}(X,\omega):=\left\{ \f \in PSH(X,\omega) \; | \; \mu_{\f}(X)=1 \right\}.
$$
For $\f \in {\mathcal E}(X,\omega)$, we set
$
MA(\f):=\mu_{\f}.
$
\end{defi}

%The notation is justified by the following important fact \cite{GZ07}:
%
%\begin{thm}
%The complex Monge-Amp\`ere operator $\f \mapsto MA(\f)$ is well defined on the class 
%${\mathcal E}(X,\omega)$: for every decreasing sequence of bounded $\omega$-psh functions $\f_j$, 
%the measures $MA(\f_j)$ converge towards $\mu_\f$, if $\f \in {\mathcal E}(X,\omega)$.
%\end{thm}
%
% 
%
%Every bounded $\omega$-psh function clearly belongs to ${\mathcal E}(X,\omega)$ 
%since  in this case $\{\f >-j\}=X$ for $j$ large enough, hence 
%$
%\mu_{\f} \equiv \mu_j=MA(\f_j)=MA(\f).
%$
%The class ${\mathcal E}(X,\omega)$ also contains many $\omega$-psh functions which are unbounded.
%%When $X$ is a compact Riemann surface ($n=\dim_{\C} X=1$),
%%the set ${\mathcal E}(X,\om)$ is the set of $\om$-sh functions
%%whose Laplacian does not charge polar sets.
%%(this is true only for $\omega$ K\"ahler right?)
%
%\begin{exa}
%If $\f \in PSH(X,\omega)$ is normalized so that $\f \leq -1$, then $-(-\f)^\e $ belongs to 
%${\mathcal E}(X,\om)$ whenever $0 \leq \e <1$.
%The functions which belong to the class ${\mathcal E}(X,\om)$,
%although usually unbounded,
%have relatively mild singularities. In particular they have zero Lelong numbers on $\Amp(\alpha)$.
%\end{exa}

%It is shown in \cite{GZ07} that the maximum principle $(MP)$ 
%and the comparison principle
%continue to hold in the class ${\mathcal E}(X,\omega)$. 
The latter can be characterized as the largest class for which the complex Monge-Amp\`ere is well defined and the maximum principle holds \cite[Theorem 1.5]{GZ07}. We further note
that the {\it domination principle} holds (\cite[Corollary 2.5]{BEGZ10}, \cite[Proposition 2.4]{DDL17}:

\begin{prop} \label{pro:dp}
If $\f,\p \in {\mathcal E}(X,\om)$ are such that 
$$
\f(x) \leq \p(x)
\text{ for } MA(\p)-\text{a.e. } x,
$$
then $\f(x) \leq \p(x)$ for all $x \in X$.
\end{prop}

%The convergence of the approximating measures
%$\mu_j=MA(\max(\f,-j))$ towards $\mu_{\f}$ holds in the sense of Borel measures: for all Borel sets $B$,
%$$
%\mu_{\f}(B):=\lim_{j \rightarrow +\infty} \mu_j(B).
%$$
%In particular when $B=P$ is a pluripolar set, we obtain $\mu_j(P)=0$, hence 
%$$
%\mu_{\f}(P)=0
%\text{ for all pluripolar sets } P.
%$$
%Conversely, one can show \cite{GZ07,BEGZ10} that a probability measure $\mu$ equals $\mu_{\f}$ for some
%$\f \in {\mathcal E}(X,\omega)$ whenever $\mu$ does not charge pluripolar sets
%(one then says that $\mu$ is non-pluripolar).

\smallskip

It follows from the $\partial\overline{\partial}$-lemma that any positive closed current 
$T \in \a$ can be written $T=\omega+dd^c \f$ for some function $\f \in PSH(X,\omega)$
which is unique up to an additive constant. 

\begin{defi}
We let ${\mathcal E}(\a)$ denote the set of all positive currents in $\a$,
$T=\omega+dd^c \f$,
with $\f \in {\mathcal E}(X,\omega)$.
\end{defi}

\noindent Note that the definition above does not depend on the choice of $\omega$, nor on the choice of  $\f$.

\subsection{The class ${\mathcal E}^1(X,\omega)$}

\subsubsection{The Aubin-Mabuchi functional}\label{AubinMabuchi_functional}

Each tangent space $T_{\f}\HH_\omega$ admits the following orthogonal decomposition
\begin{equation*}
    T_{\f}\HH_\omega = \set{\psi \in \Cinf(X) ; \: \beta_{\f}(\psi) = 0}\oplus \Real ,
\end{equation*}
where $\beta=MA$ is the 1-form defined on $\HH$ by
\begin{equation*}
    \beta_{\f}(\psi) = \int_{X}\psi \,MA(\f).
\end{equation*}

It is a classical observation due to Mabuchi  that the 1-form $\beta$ is closed. 
Therefore there exists a unique function $E$ defined on the convex open set $\HH_\omega$, such that $\beta = dE$ and $E(0) = 0$. It is often called the \emph{Aubin-Mabuchi functional} and can be expressed (after integration along affine paths) by
\begin{equation*}
%\label{equ:Aubin_Mabuchi}
    E(\f) = \frac{1}{(n+1)V_\a} \sum_{j=0}^{n} \int_{X}\f \,  
(\omega+dd^c  \f)^{j}  \wedge   \omega^{n-j}   .
\end{equation*}

\begin{lem}  \label{lem:AubinMabuchi1}
The Aubin-Mabuchi functional $E$ is concave along euclidean segments, non-decreasing, and 
satisfies the cocycle condition
$$
E(\f)-E(\p)=\frac{1}{(n+1)V_\a} \sum_{j=0}^{n} \int_{X} (\f-\p) \,  
(\omega+dd^c  \f)^{j}  \wedge   (\omega+dd^c \p)^{n-j} 
$$

It is {affine} along geodesics and convex along subgeodesics in $\HH_\omega$.
\end{lem}

%We sketch the proof for the reader's convenience.

\begin{proof}
These properties are well-known when $\omega$ is a K\"ahler form.

The monotonicity property follows from the definition since the first derivative of $E$
is $dE=\beta=MA \geq 0$, a probability measure: if $\f_t$ is an arbitrary path, then
$$
\frac{d}{dt}E(\f_t)=\int_X \dot{\f}_t MA(\f_t).
$$

It follows from Stokes theorem that
\begin{eqnarray*}
    \frac{d^{2}}{dt^{2}} E(\f_{t}) &= & \int_{X} \ddot{\f_t}\,MA(\f_t) 
+\frac{n}{V_\a} \int_{X} \dot{\f_t} \, dd^c \dot{\f_t} \wedge \omega_\f^{n-1}  \\
&=&   \int_{X} \left\{ \ddot{\f_t}\,MA(\f_t) -\frac{n}{V_\a} d \dot{\f_t}  \wedge d^c \dot{\f_t} \wedge \omega_{\f_t}^{n-1}  \right\}.
\end{eqnarray*}
Thus $E$ is concave along euclidean segments ($\ddot{\f}_t=0$), affine along Mabuchi geodesics,
and convex along Mabuchi subgeodesics.
The cocycle condition follows by differentiating $E(t\f+(1-t)\p)$.

These computations are mereley heuristic as $t\rightarrow \varphi_t(x)$ is poorly regular when $\f_t$ is a geodesic or a subgeodesic.
We can however approximate $\omega$ by $\omega_\e=\omega+\e \omega_X$, consider $(\f_t^\e)$ the corresponding geodesic
\begin{equation}\label{energy_approx} 
E_{\omega_\varepsilon} (\f_t^\varepsilon)=\frac{1}{(n+1)V_\varepsilon} \sum_{j=0}^n \int_X \f_t^\varepsilon (\omega_\varepsilon+dd^c \f_t^\varepsilon)\wedge \omega_\varepsilon^{n-j}.
\end{equation}
It follows from Proposition \ref{cor:approx}
that $\e \mapsto \f_t^\e$ decreases to $\f_t$, hence
$t \mapsto E(\f_t)$ is affine, being the limit of the affine maps
$t \mapsto E_{\omega_\varepsilon} (\f_t^\varepsilon)$.
 
%We can similarly approximate an euclidean segment $t \f_1+(1-t)\f_0$ by smoother euclidean segments $t\f_1^\varepsilon+(1-t)\f_0^\varepsilon$ in $\psh(X, \omega_\varepsilon)$ in order to justify that $t\rightarrow E(\f_t)$ is concave.
For subgeodesics we approximate again $\omega$ by $\omega_\varepsilon$ and we proceed as in the K\"ahler case.
\end{proof}

Observe that
$
    E(\f + t) =   E(\f)+t.
$
Given $\f\in \HH_\omega$ there exists a unique  $c\in\Real$ such that $E(\f + c) = 0$. The restriction of the Mabuchi metric to the fiber $E^{-1}(0)$ induces a Riemannian structure on the quotient space 
$\HH_{\a} = \HH_\omega / \Real$ and allows to decompose
$
    \HH_\omega = \HH_{\a} \times \Real
$
as a product of Riemannian manifolds.

\begin{defi}
For $\f \in PSH(X,\omega)$, we set
$$
E(\f):=\inf \{ E(\p) \, ; \, \f \leq \p \text{ and } \p \in PSH(X,\omega) \cap L^{\infty}(X) \} \in [-\infty,+\infty[
$$
and
$
{\mathcal E}^1(X,\omega):=\{ \f \in PSH(X,\omega) \, ; \, E(\f)>-\infty\}.
$
\end{defi}

\subsubsection{Strong topology on  ${\mathcal E}^1(\a)$}

%The class $\E^1(X,\omega)$  provides the natural frame for the variational approach to studying complex 
%Monge-Amp\`ere operators \cite{BBGZ}. 
Set
$$
I(\f,\p)=\int_X (\f-\p) \left( MA(\p)-MA(\f) \right).
$$

It has been shown in \cite{BBEGZ} that $I$ defines a complete metrizable uniform structure on
${\mathcal E}^1(\a)$. More precisely we identify ${\mathcal E}^1(\a)$ with the set
$$
\E^1_{norm}(X,\omega)=\{ \f \in \E^1(X,\omega) \, | \, \sup_X \f=0 \}
$$
of normalized potentials. Then
\begin{itemize}
\item $I$ is symmetric and positive on $\E^1_{norm}(X,\omega)^2 \setminus \{ \rm{diagonal} \}$;
\item $I$ satisfies a quasi-triangle inequality \cite[Theorem 1.8]{BBEGZ};
\item $I$ induces a uniform structure which is metrizable \cite{Bourbaki};
\item the metric space $(\E^1(\a),d_I)$ is complete \cite[Proposition 2.4]{BBEGZ}, where $d_I$ denotes one of the distances induced by the uniform structure $I$.
\end{itemize}

\begin{defi}
The {\it strong topology} on $\E^1(\a)$ is the  metrizable topology defined by $I$. 
\end{defi}

The corresponding notion of convergence is
the {\it convergence in energy} previously introduced
in \cite{BBGZ} (see \cite[Proposition 2.3]{BBEGZ}). It is the coarsest refinement of the weak topology
such that $E$ becomes continuous. In particular 
$\text{ if } T_j \longrightarrow T  \text{ in } (\E^1(\a),d_I)$, then
$$
T_j \longrightarrow T  \; \text{ weakly }     \text{ and } \; 
T_j^n \longrightarrow T^n
$$ 
in the weak sense of Radon measures, while the Monge-Amp\`ere operator is usually discontinuous for the weak topology of currents.

%The key point in order to understand this notion of convergence is the following inequality (see \cite[eq. (2.7)]{BBGZ}):
%$$
%\frac{1}{n+1} I(\f_j,\f) \leq F_{MA(\f)}(\f)-F_{MA(\f)}(\f_j) \leq I(\f_j,\f),
%$$
%where
%$
%F_{\mu}(\f):=E(\f)-\int_X \f \, d\mu
%$
%is a functional associated to a finite energy probability measure $\mu$, that it maximized 
%precisely when $\f=\f_\mu \in \E^1_{norm}(X,\omega)$ is the normalized Monge-Amp\`ere potential of $\mu=MA(\f_\mu)$.
%
%Thus for normalized potentials $\f_j \in \E^1_{norm}(X,\omega)$, one has $I(\f_j,\f) \rightarrow 0$
%if and only if they form a maximizing sequence for $F_\mu$, hence they strongly converge
%towards $\f_\mu$, as follows from \cite[Theorem 4.7]{BBGZ}.

%\begin{exa}
%When $\dim_\C X=n=1$, $\E^1(X,\omega)=PSH(X,\omega) \cap W^{1,2}(X)$ is the set of 
%$\omega$-subharmonic functions with square integrable gradient. The strong topology on
%$\E^1(\a)$ is the one induced by the Sobolev norm. 
%When $X=\P^1_\C$ is the Riemann sphere, $\omega=\omega_{FS}$ is the Fubini-Study K\"ahler form,   
%and $\e_j \searrow 0^+$, $C_j \nearrow +\infty$, the functions
%$$\f_j([z])=\e_j \max ( \log ||z||, \log |z_0|-C_j)-\e_j \log ||z|| \in \E^1(X,\omega)$$
%converge to zero in $L^2(X)$ but not in the strong topology if $\e_j^2 C_j \geq 1$ since
%$$I(\f_j,0)=\int_{\P^1} d \f_j \wedge d^c \f_j \sim \e_j^2 C_j.$$
%\end{exa}

\subsubsection{Yet another distance}

To fit in with the notations of the next section, we introduce yet another notion of convergence
in ${\mathcal E}^1(X,\omega)$.
%For $\f,\p \in \E^1(X,\omega)$, 
We set
$$
I_1(\f,\p):= \int_X |\f-\p| \left[ \frac{MA(\f)+MA(\p)}{2} \right] 
$$

This symmetric quantity is non-negative. It follows from Proposition \ref{pro:dp} that 
it only vanishes on the diagonal of $\E^1(X,\omega)^2$, while Theorem \ref{thm:equiv} will insure that it satisfies a quasi-triangle inequality. Hence, $I_1$ induces a uniform structure which is metrizable \cite{Bourbaki}.\\
For $C>0$, we set
$$
\E_C^1(X,\omega):=\{ \f \in \E^1(X,\omega) \, ; \, E(\f) \geq -C \text{ and } \f \leq C \}.
$$
It follows from Hartogs lemma, the upper-semi continuity and the conca\-vi\-ty of $E$ along euclidean segments (Lemma \ref{lem:AubinMabuchi1}) that this set is a compact and convex subset
of $PSH(X,\omega)$, when endowed with the $L^1$-topology (see \cite[Lemma 2.6]{BBGZ}).

\begin{prop}\label{prop: I and I_1}
For all $\f,\p \in \E^1(X,\omega)$, $I(\f,\p) \leq 2 I_1(\f,\p)$.
Conversely for each $C>0$, there exists $A>0$ such that for all 
$\f,\p \in \E_C^1(X,\omega)$
\begin{equation}\label{$I_1$-convergence}
I_1(\f,\p) \leq \int_X \left[ 2 \max(\f,\p)-(\f+\p) \right] MA(0) +A \, I(\f,\p)^{1/2^n}.
\end{equation}

In particular the topologies induced by $I,I_1$ on $\E_{norm}^1(X,\omega)$ are the same.
\end{prop}

Observe that $I_1$ induces a distance on $\E^1(X,\omega)$, but $I$ is merely defined
on $\E_{norm}^1(X,\omega)$, as $I(\f+c,\p+c')=I(\f,\p)$, for any $c,c' \in \R$.

\begin{proof}
The first inequality is obvious, as
$$
I(\f,\p)=\int_X (\f-\p) \left( MA(\p)-MA(\f) \right) \leq \int_X |\f-\p| \left( MA(\p)+MA(\f) \right).
$$

It follows from Proposition \ref{prop:python} below  that
$$
I_1(\f,\p)=I_1(\f,\max(\f,\p))+I_1(\max(\f,\p),\p),
$$
hence it suffices to establish the second inequality when  $\f \leq \p$. In this case
$$
I_1(\f,\p) \leq \int_X (\p-\f) MA(\f),
$$
by Lemma \ref{lem:controleL2},  while Cauchy-Schwarz inequality yields
\begin{eqnarray*}
\lefteqn{
\int_X (\p-\f) MA(\f)  } \\ 
&=& \int_X (\p-\f) MA(0)+\int_X d(\f-\p) \wedge d^c \f \wedge S_\f \\
&\leq & \int_X (\p-\f) MA(0)+I(\f,0)^{1/2} \left( \int_X d(\f-\p) \wedge d^c (\f-\p) \wedge S_\f \right)^{1/2},
\end{eqnarray*}
where we have set $S_\f:=\sum_{j=0}^{n-1} \omega_\f^j \wedge \omega^{n-1-j}$.
Observing that $S_\f \leq 2^{n-1} \omega_{\f/2}^{n-1}$, we can invoke \cite[Lemma 1.9]{BBEGZ} to obtain

$$
\int_X d(\f-\p) \wedge d^c (\f-\p) \wedge S_\f  \leq c_n I(\f,\p)^{1/{2^{n-1}}} 
\left\{ I \left(\f,\frac{\f}{2} \right)^{1-1/{2^{n-1}}}+I\left(\p,\frac{\f}{2} \right)^{1-1/{2^{n-1}}} \right\}.
$$

Now $I(\f,\f/2) \leq a_n I(\f,0) \leq C'$ and \cite[Theorem 1.3]{BBEGZ} yields
$$
I(\p,\f/2) \leq b_n \left\{ I(\p,0)+I(\f/2,0) \right\}
\leq b_n' \left\{ I(\p,0)+I(\f,0) \right\} \leq C''.
$$
We thus get \eqref{$I_1$-convergence}.

In order to prove the last statement we need to show that given a sequence $\f_j \in \E_{norm}^1(X,\omega)$ converging to $\psi$ w.r.t $I$ then it converges to $\psi$ also w.r.t $I_1$, and viceversa. We first note that the $I$-convergence implies the $L^1$-convergence of the potentials \cite[Theorem 10.37]{GZ17}. This insures that
$$\int_X  \left[ 2 \max(\f_j,\p)-(\f_j+\p) \right] MA(0) \rightarrow 0 \qquad  {\rm{as}}\, j\rightarrow +\infty,$$ 
and moreover we have that $\f_j, \psi\in \E_C^1(X,\omega)$ for some $C>0$ (\cite[Lemma 10.33 and Definition 10.34]{GZ17}).
The $I_1$-convergence would then follow from \eqref{$I_1$-convergence}. Moreover, since $I(\f_j, \psi)\leq 2 I_1(\f_j, \psi)$, we conclude that the $I_1$-convergence implies the $I$-convergence.
\end{proof}

\subsection{The complete metric spaces ${\mathcal E}^p(\a)$} \label{subsection: E^p}

%\subsubsection{Finite energy classes}

Fix $p\geq 1$. Following \cite{GZ07,BEGZ10} we consider the following finite energy classes:

\begin{defi}
We set
$$
{\mathcal E}^p(X,\om):=\left\{ \f \in {\mathcal E}(X,\om) \, / \, 
 |\f|^p \in L^1( MA(\f)) \right\}
$$
and let
$
{\mathcal E}^p(\a)=\{ T=\omega+dd^c \f \, | \, \f \in {\mathcal E}^p(X,\omega) \}
$
denote the corresponding sets of finite energy currents.
\end{defi}

%\smallskip

%We list a few   properties of these classes extracted from \cite{GZ07,BEGZ10}:
%\begin{itemize}
%\item for all $\f \in {\mathcal E}^p(X,\om)$,
%$\nabla_\omega \f \in L^2(\om^n)$;
%\item $\f \in \E^p(X,\omega)$ if and only if for any (resp. one)  sequence of bounded
%$\omega$-functions decreasing to $\f$, 
%$\sup_j \int_X (-\f_j)^p MA(\f_j) <+\infty$.
%\item if  $\f_j$ is a sequence of  $\omega$-functions decreasing to $\f \in \E^p(X,\omega)$, then
%the Radon measures $(-\f_j)^p MA(\f_j)$ weakly converge to $(-\f)^p MA(\f)$.
%\item  There is $C_{p}>0$ s.t. for all 
%$0 \geq  \f_0,\ldots,\f_n \in PSH(X,\om) \cap L^{\infty}(X)$,
%$$
%0 \leq \int_X (-\f_0)^p \, \om_{\f_1} \wedge \cdots \wedge \om_{\f_n}
%\leq C_{p} \max_{0 \leq j \leq n} \left[ \int_X (-\f_j)^p \,  \om_{\f_j}^n \right].
%$$
%In particular the class $\E^p(X,\omega)$ is  convex.
%\end{itemize}

%\subsubsection{Strong topology on  ${\mathcal E}^p(\a)$}\label{subsection: E^p}

On the class ${\mathcal E}^p(X, \omega)$, $p \geq 1$,
we define
%For $\f,\p \in \E^p(X,\omega)$, we set
$$
I_p(\f,\p):=\left( \int_X |\f-\p|^p \left[ \frac{MA(\f)+MA(\p)}{2} \right] \right)^{1/p}
$$
This quantity is well-defined   by \cite[Proposition 3.6]{GZ07}. It is obviously non-negative and 
symmetric. It follows from the domination principle (Proposition \ref{pro:dp}) that
$$
I_p(\f,\p)=0 \Longrightarrow \f=\p.
$$  
Moreover, it will follow from Theorem \ref{thm:equiv} (which
shows in particular that $I_p$ satisfies a quasi-triangle inequality) that $I_p$ induces a uniform structure. We can then define the following:

\begin{defi}
The strong topology on $\E^p(\a)$ is the one induced by $I_p$. 
\end{defi}

By \cite[Theorem 2.17]{BEGZ10}, a decreasing sequence converges strongly. We also have good convergence properties
if we approximate by  slightly larger finite energy classes $\E^p(X,\omega_\e)$:

\begin{prop}  \label{prop:approx2}
Fix  $\omega_\e=\omega+\e \omega_X$, $\e>0$. If  $\f,\p  \in \E^p(X,\omega) \cap L^{\infty}(X)$, then
$\f,\p  \in \E^p(X,\omega_\e)\cap L^{\infty}(X)$ and
$I_{p,\omega_\e}(\f,\p) \rightarrow I_{p,\omega}(\f,\p)$ as $\e \rightarrow 0$.

Moreover, if $\f,\p  \in \E^p(X,\omega)$ and $\f_j,\psi_j$ are sequences of smooth $\omega_{\e_j}$-psh functions decreasing to $\f,\psi$ 
with $\e_j \rightarrow 0$, then
$$
I_{p, \omega_{\varepsilon_j} }(\f_j, \psi_j) \rightarrow I_{p, \omega}(\f, \psi)
$$
as $j$ goes to $+\infty$.
\end{prop}

\begin{proof}
Note that $\f, \p$ belong to any energy class w.r.t any K\"ahler form since they are bounded. In particular $\f, \p\in \E^p(X, \omega_\e)$. The first assertion follows from the fact that $(\omega_\varepsilon+dd^c \varphi)^n$ and $(\omega_\varepsilon+dd^c \p)^n$ 
converges weakly to $(\omega+dd^c \varphi)^n$ and $(\omega+dd^c \p)^n$ as $\e \rightarrow 0$, respectively. For the second statement, we observe that by symmetry it suffices to prove that
$$\int_X |\f_j-\p_j|^p (\omega_{\varepsilon_j}+dd^c \f_j)^n \rightarrow \int_X |\f-\p|^p (\omega+dd^c \f)^n, \quad {\rm as}\; j\rightarrow +\infty.$$
Given a bounded function $f$ on $X$, we set
$$|f|_p:=\left(\int_X |f|^p (\omega_{\varepsilon_j}+dd^c \f_j)^n\right)^{1/p}.$$
The triangle inequality yields
$$|\f_j-\p_j|_p\leq |\f-\p|_p+|(\f_j-\f)|+ |(\p-\psi_j)|_p$$
and similarly
$$  |\f_j-\p_j|_p\geq |\f-\p|_p-|(\f_j-\f)|-| (\p-\psi_j)|_p .
$$
Since $\f-\psi$ is a positive quasi-continuous uniformly bounded function on $X$, it follows from \cite[Theorem 4.26]{GZ17} that

%the continuity of the Monge-Amp\`ere operator along decrea\-sing sequence \cite[Theorem 1.9]{GZ07} and \cite[Corollary 1.14]{Kol05} that

 $$ |\f-\p|_p^p=\int_{X} |\f-\psi|^p (\omega_{\varepsilon_j}+dd^c \f_j)^n \rightarrow  \int_{X} |\f-\psi|^p (\omega+dd^c \f)^n$$
 as $j\rightarrow +\infty$. Moreover, we claim that the terms 
$|(\f_j-\f)|_p$ and $ |(\p-\psi_j)|_p$ goes to  $ 0$  as $j\rightarrow +\infty.$
 Lemma \ref{lem:controleL2} together with the fact that $\omega_{\varepsilon_j}\leq \omega+\omega_X$ yields
$$\int_X (\f_j-\f)^p (\omega_{\varepsilon_j}+dd^c \f_j)^n  \leq \int_X (\f_j-\f)^p (\omega+\omega_X+dd^c \f)^n.  $$
Note that $\f_j, \f\in \E^p(X, \omega+\omega_X)$ (since they are bounded). Hence \cite[Theorem 3.8]{GZ07} insures that the integral at the RHS of the above inequality is finite.\\
Since $\f_j$ is decreasing to $\f$, it then follows from the dominated convergence theorem that $|(\f_j-\f)|_p^p\rightarrow 0$ as $j\rightarrow +\infty$. Fix $j_0<j$. Then
$$\int_X  (\p_j-\p)^p (\omega_{\varepsilon_j}+dd^c \f_j)^n \leq \int_X  (\p_{j_0}-\p)^p (\omega+\omega_X+dd^c \f_j)^n.$$
It follows again from the continuity of the Monge-Amp\`ere operator along decrea\-sing sequence, \cite[Corollary 1.14]{Kol05} and the dominated convergence theorem that letting $j\rightarrow +\infty$ and then $j_0\rightarrow +\infty$ we get
$$\int_X  (\p_{j_0}-\p)^p (\omega+\omega_X+dd^c \f_j)^n\rightarrow 0 .$$
Thus $|(\p_j-\p)|_p^p\rightarrow 0$ as $j\rightarrow +\infty$. Hence the conclusion.
\end{proof}

It follows from H\"older inequality
that the strong topology on ${\mathcal E}^p(\a)$
 is stronger than the one on $\E^1(\a)$: if a sequence $(\f_j) \in \E^p(X,\omega)$ is a Cauchy sequence for $I_p$, then it is
a Cauchy sequence in $(\E^1(X,\omega),d_I)$ since
$$
0 \leq I(\f,\p)=\int_X (\f-\p) \left[ MA(\p)-MA(\f) \right]\leq {2}^{1/p} I_p(\f,\p).
$$
Since $(\E^1(X,\omega),d_I)$ is complete, there is $\f \in \E^1(X,\omega)$
s.t. $d_I(\f_j,\f) \rightarrow 0$.
Now $I_p(\f_j,0)$ is bounded and $MA(\f_j)$ converges to $MA(\f)$
(by \cite[Proposition 5.6]{BBGZ}). Thus  $\f \in \E^p(X,\omega)$ by Fatou's and Hartogs' lemma.

One would now like to prove that $I_p(\f_j,\f) \rightarrow 0$ and conclude that the space 
$(\E^p(X,\omega),I_p)$ is complete,
arguing as in \cite[Proposition 2.4]{BBEGZ}.
%There is an abuse of terminology here as we haven't  checked that $I_p$ induces a uniform structure.
%This follows from Theorem \ref{thm:equiv} which
%shows in particular that $I_p$ satisfies a quasi-triangle inequality 
%(like $I$ does, see \cite[Theorem 1.8]{BBEGZ}).
We refer the reader to Theorem \ref{thm_Completeness} for a neat treatment.

\begin{lem} \label{lem:controleL2}
Let $\f,\p$ be bounded $\omega$-psh functions and $S$ be a positive closed current of bidimension $(1,1)$
on $X$. If $\f \leq \p$, then
$$
\int_X (\p-\f)^p \omega_\p \wedge S \leq \int_X (\p-\f)^p \omega_\f \wedge S.
$$
In particular 
%for all $0 \leq j \leq n$,
$
V_\a^{-1} \int_X (\p-\f)^p \omega_\p^j \wedge \omega_\f^{n-j}
\leq \int_X (\p-\f)^p MA(\f).
$
\end{lem}

\begin{proof}
By Stokes' theorem,
$$
\int_X (\p-\f)^p \omega_\f \wedge S-\int_X (\p-\f)^p \omega_\p \wedge S
=p\int_X (\p-\f)^{p-1} d(\f-\p) \wedge d^c (\f-\p) \wedge S
$$
is non-negative if $(\p-\f) \geq 0$. 

The second assertion follows by applying the first one  inductively.
\end{proof}

 We now establish a few useful properties of $I_p$ that will notably allow to compare
 $I_p$ to $d_p$ in the next section.

\begin{prop} \label{prop:python}
For $\f,\p \in \E^p(X,\omega)$,
$$
I_p(\f,\p)^p=I_p(\f,\max(\f,\p))^p+I_p(\max(\f,\p),\p)^p.
$$
\end{prop}

\begin{proof}
Recall that the maximum principle insures that
$$
\1_{\{ \f<\p\}} MA(\max(\f,\p))= \1_{\{ \f<\p\}}MA(\p),
$$
while $(\f-\max(\f,\p))^p=0$ on $(\f \geq \p)$, thus
$$
2 I_p(\f,\max(\f,\p))^p=\int_{\{\f<\p \}} |\f-\p|^p \left[ MA(\f)+MA(\p)\right].
$$
Similarly
$
2 I_p(\p,\max(\f,\p))^p=\int_{\{\f>\p \}} |\f-\p|^p \left[ MA(\f)+MA(\p)\right]
$
and the result follows since 
$$
I_p(\f,\p)^p=\frac{1}{2}\int_{\{\f \neq \p \}} |\f-\p|^p \left[ MA(\f)+MA(\p)\right].
$$
\end{proof}

\begin{cor}
For all $\f,\p \in \E^p(X,\omega)$,
$$
I_p\left( \frac{\f+\p}{2},\p \right) \leq I_p(\f,\p).
$$
\end{cor}

 \begin{proof}
 By approximating $\f,\p$ from above by a decreasing sequences, it suffices to treat the case when $\f,\p \in \H_{\omega}$.
 Changing $\omega$ in $\omega_\p$, we can further assume that $\p=0$. It follows from Proposition \ref{prop:python}
 that
 $$
 I_p\left( 0,\f/2 \right)^p = I_p(0,\max(0,\f/2))^p+I_p(\max(0,\f/2), \f/2)^p.
 $$
 It follows from Lemma \ref{lem:controleL2} that 
 \begin{eqnarray*}
  I_p(0,\max(0,\f/2))^p &\leq& \int_X \max(0,\f/2)^p MA(0)  \\
  &=&2^{-p} \int_X \max(0,\f)^p MA(0) 
  \leq I_p(0,\max(0,\f))^p.
 \end{eqnarray*}
 
 We claim that for all $0 \leq j \leq n$,
 $$
 \int_X (\max(0,\f)-\f)^p \omega_\f^j \wedge \omega^{n-j} \leq \int_X (\max(0,\f)-\f)^p \omega_\f^n.
 $$
 Assuming this for the moment, it follows again from Lemma \ref{lem:controleL2} that 
 {\small
\begin{eqnarray*}
I_p(\max(0,\f/2), \f/2)^p & \leq & \int_X (\max(0,\f/2)-\f/2)^p MA(\f/2) \\
&=& \frac{1}{2^{n+p} V_\a} \sum_{j=0}^n C_n^j \int_X (\max(0,\f)-\f)^p \omega_\f^j \wedge \omega^{n-j} \\
 &  \leq & \frac{1}{2} \int_X (\max(0,\f)-\f)^p MA(\f) 
\leq   I_p(\f,\max(0,\f))^p.
\end{eqnarray*}
}
We infer
 $$
 I_p\left( 0,\f/2 \right)^p  \leq  I_p(0,\max(0,\f))^p+I_p(\max(0,\f), \f)^p=I_p(0,\f)^p,
 $$
 by using Proposition \ref{prop:python} again. 

\smallskip

It remains to justify our claim. Set $S=\omega^{j-1} \wedge \omega_\f^{n-j}$.
 It suffices, by induction, to establish the following inequality:
 \begin{eqnarray*}
 \lefteqn{ \int_X (\max(0,\f)-\f)^p \omega \wedge S } \\
 &=&  \int_X (\max(0,\f)-\f)^p \omega_\f \wedge S - \int_X (\max(0,\f)-\f)^p dd^c \f \wedge S \\
 & \leq  & \int_X (\max(0,\f)-\f)^p \omega_\f \wedge S.
 \end{eqnarray*}
This follows by observing that

 \begin{eqnarray*}
- \int_X (\max(0,\f)-\f)^p dd^c \f \wedge S
 &=& p \int_X (\max(0,\f)-\f)^{p-1} d (\max(0,\f)-\f) \wedge d^c \f \wedge S \\
 &=& -p \int_{\{\f<0\}}  (-\f)^{p-1} d \f \wedge d^c \f \wedge S \leq 0.
 \end{eqnarray*}
 
 \end{proof}

\section{Comparing distances} \label{sec:comparison}
 
 In this section we show that $I_p$ is equivalent to $d_p$
 (Theorem \ref{thm:equiv}). Recall that:
    $$\mathcal{H}_{bd}:=\{\f\in \psh(X, \omega) \cap L^\infty(X),\; \f=P_{\omega}(f)\; {\rm for\, some}\; f\in C^0(X)\, {\rm with}\, dd^c f\leq C\omega_X,\; C>0\}.$$
 In the following we are going to use several times and in a crucial way that Theorem \ref{thm:magicFormulaBdd} insures
 $${d}_p^p(\f_0, \f_1)= \int_X |\dot{\f}_0|^p \frac{(\omega+dd^c \f_0)^n}{V}= \int_X |\dot{\f}_1|^p \frac{(\omega+dd^c \f_1)^n}{V},\quad \forall \f_0, \f_1\in \H_{bd}.$$

 %(Section \ref{sec:regularity}) 
 
\subsection{Kiselman transform and geodesics}

Let $(\f_t)_{0 \leq t \leq 1}$ be the Mabuchi geodesic.
For all $x \in X$, $t \in [0,1] \mapsto \f_t(x) \in \R$ is convex. It is natural to consider
its Legendre transform $u_s(x): s \mapsto \sup_{t \in [0,1]} \{ st-\f_t(x)\}$.
This function is convex in $s$, but the dependence in $x$ is $-\omega$-psh, so 
we rather consider $-u_s$. We finally change $s$ in $-s$ to obtain a more elegant formula,
$$
\p_s(x):=\inf_{0 \leq t \leq 1} \{ st+\f_t(x)\}.
$$

\begin{prop}  \label{pro:kiselman}
The functions $x \mapsto \p_s(x)$ are $\omega$-plurisubharmonic.
In particular
$
x \mapsto \p_0(x)=\inf_{0 \leq t \leq 1} \f_t(x)
\text{ is } \omega$-psh.
\end{prop}

This is the minimum principle of Kiselman \cite{Kis78}.
For $\f_0,\f_1 \in \H_{bd}$ we let $\f_0 \vee \f_1$ denote the greatest $\omega$-psh function that lies below
$\f_0$ and $\f_1$. In the notations of Berman-Demailly \cite{BD12}
$$
\f_0 \vee \f_1=P(\min(\f_0,\f_1)),
$$
while $\f_0 \vee \f_1$ is denoted by $P(\f_0,\f_1)$ in \cite{Dar14}.

An important consequence of Kiselman minimum principle \cite{Kis78}
is the following observation due to  Darvas and Rubinstein \cite{DR14}:

\begin{prop} \label{pro:kis}
The function $\f_0 \vee \f_1$ is a bounded $\omega$-psh which has locally bounded Laplacian on the ample locus of $\alpha=\{\omega\}$  and 
 its Monge-Amp\`ere measure $MA(\f_0 \vee \f_1)$ is supported on the coincidence set
$$
\{ x \in X \, | \, \f_0 \vee \f_1(x)=\min(\f_0,\f_1)(x) \}.
$$
Moreover 
$
MA(\f_0 \vee \f_1)=\1_{\{\f_0 \vee \f_1=\f_0\}} MA(\f_0)+\1_{\{\f_0 \vee \f_1=\f_1<\f_0\}} MA(\f_1).
$

\smallskip

Let $(\f_t)$ be the Mabuchi geodesic joining $\f_0$ and $\f_1$.
Then for all $x \in X$,
$$
\f_0 \vee \f_1(x)=\inf_{t \in [0,1]} \f_t(x).
$$
\end{prop}

\begin{proof}
%The arguments in \cite[Section 2.2]{EGZ16} insures that if $h$ is a smooth function on $X$, then $P(h)$ is a $\omega$-psh function, continuous on $\Amp(\{\omega\})$. The same result holds if $h$ is merely continuous. Indeed, let $h_j$ be a sequence of smooth functions on $X$ converging uniformly to $h$. For any compact set $K\subset \Amp(\{\omega\})$, we have $\|P(h_j)-P(h)\|_{L^\infty(K)}\leq \|h_j-h\|_{L^\infty(X)}$. Hence $P(h_j)$ converges uniformly to $P(h)$ in the ample locus of $\{\omega\}$, and so $P(h)$ is a $\omega$-psh function that is continuous on $\Amp(\{\omega\})$. 
%The continuity of $\f_0 \vee \f_1$ on $\Amp(\{\omega\})$ follows from \cite[Section 2.2]{EGZ16} together with the same arguments at the end of Proposition \ref{prop:Semmes}.

It follows from a classical balayage procedure that goes back to Bedford and Taylor \cite{BT82} that
$MA(\f_0 \vee \f_1)$ is supported on the coincidence set 
$
\{ x \in X \, | \, \f_0 \vee \f_1(x)=\min(\f_0,\f_1)(x) \}
$
This holds true more generally for the Monge-Amp\`ere measure of any 
envelope, namely  $$\1_{\{P(h)<h\}} MA(P(h)) \equiv 0,$$
where $h$ is a bounded lower semcontinuous function.

We have observed in Proposition \ref{pro:kiselman} that
$
x \mapsto \inf_{t \in [0,1]} \f_t(x)
$
is a $\omega$-psh function.
Since it lies both below $\f_0$ and $\f_1$, we infer 
$$
\inf_{t \in [0,1]} \f_t \leq \f_0 \vee \f_1.
$$
Conversely $(t,x) \mapsto \f_0 \vee \f_1(x)$ is a subgeodesic (independent of $t$), hence
for all $t,x$, $ \f_0 \vee \f_1(x) \leq \f_t (x)$. Thus $\p:=\f_0 \vee \f_1=\inf_{t \in [0,1]} \f_t$, hence $\p$ is bounded thanks to Proposition \ref{prop:Semmes}.

By Proposition \ref{pro:kiselman}, $\psi$ is $\omega$-psh, hence $A\omega_X$-psh for some K\"ahler form $\omega_X$ and $A>0$. Thus
 $\sup_X \Delta_{\omega_X} \p \geq -C$ for some $C>0$.

It follows from the work of Berman and Demailly \cite{BD12} (see also \cite[Theorem 1.2]{Ber}) that 
for any compact subset $K \subset \Amp(\a)$, there exists 
$C_K>0$ such that for all $t \in [0,1]$,
$$
\sup_K \Delta_{\omega_X} \f_t < C_K n.
$$
Thus $(-\f_t)$ is a family of $C_K\omega_X$-psh functions in a neighborhood of $K$,
which are uniformly bounded from above. Thus 
$$
-\p=\sup_{0 \leq t \leq 1} (-\f_t)=-\inf_{0 \leq t \leq 1}\f_t
$$
is $C_K\omega_X$-psh near $K$, in particular $\Delta_{\omega_X} \p < C_Kn$. This means that $\psi$ has locally bounded laplacian on $\Amp(\alpha)$.

It follows then from classical arguments that the measure $MA(\f_0 \vee \f_1)$
is absolutely continuous with respect to Lebesgue measure.
%When $\{\f_0=\f_1\}$ is a smooth hypersurface, it follows from  \cite[Proposition 4.5]{DR14} that $\{\f_0=\f_1\}$ does not intersect
%the coincidence set. The latter thus decomposes as a disjoint union
%$$\{\f_0 \vee \f_1=\f_0 \} \sqcup \{\f_0 \vee \f_1=\f_1 \}.$$
Since $\f_0 \vee \f_1,\f_0$ (resp. $\f_0 \vee \f_1, \f_1$) have locally bounded Laplacian in
$\Amp(\a)$, it follows from \cite[Lemma 7.7]{GT83} that their second partial derivatives
agree on $\{\f_0 \vee \f_1=\f_0 \}$ (resp. on $\{\f_0 \vee \f_1=\f_1 \}$), hence
$$
MA(\f_0 \vee \f_1)=\1_{\{\f_0 \vee \f_1=\f_0\}} MA(\f_0)+\1_{\{\f_0 \vee \f_1=\f_1<\f_0\}} MA(\f_1).
$$
We have used here the fact that none of the measures 
$MA(\f_0 \vee \f_1),MA(\f_0)$, $MA(\f_1)$ charges the pluripolar set $X \setminus \Amp(\a)$.
%When $\{\f_0=\f_1\}$ is not a  smooth hypersurface, we replace $\f_1$ by $\f_1+\e$, with $\e>0$. It follows from Sard's theorem that 
%$\{\f_0=\f_1+\e_j\}$ is a smooth hypersurface for a sequence $\e_j$ decreasing to $0^+$. The previous analysis yields
%\begin{equation} \label{eq:kis}
%MA(\f_0 \vee (\f_1+\e))=\1_{A_\e} MA(\f_0)+\1_{B_\e} MA(\f_1),
%\end{equation}
%where $A_\e=\{\f_0 \vee (\f_1+\e) =\f_0\}$ decreases to $A_0$, while
%$$B_\e=\{\f_0 \vee (\f_1+\e) =\f_1+\e\}=\{(\f_0-\e) \vee \f_1 =\f_1\}$$
%increases to $B_0$. It follows therefore from the monotone convergence theorem that the 
%RHS of (\ref{eq:kis}) converges to $\1_{A_0} MA(\f_0)+\1_{B_0} MA(\f_1)$, while the LHS converges to
%$MA(\f_0 \vee \f_1)$ by \cite{BT82}, since $\f_0 \vee (\f_1+\e_j)$ decreases to $\f_0 \vee \f_1$,
%as $j \rightarrow +\infty$.
\end{proof}

A basic observation that we shall use on several occasions is the following: 

\begin{lem} \label{lem:maxprince}
Assume $\f_0,\f_1 \in  \H_{bd}$ and let $(\f_t)_{0 \leq t \leq 1}$ be the Mabuchi geodesic joining $\f_0$ to $\f_1$. Then:
$$d_p(\f_0,\f_1) \leq ||\f_1-\f_0||_{L^{\infty}(X)}.$$
Moreover,
\begin{itemize}
\item[(i)]  If $\f_0(x)\leq \f_1(x)$ for some $x \in X$, then $\dot{\f}_1(x) \geq 0$.
\item[(ii)]  If $\f_0(x)\leq \f_1(x)$ for all $x\in X$ then $\dot{\f}_t(x) \geq 0$ for all $x\in X $ and a.e $t\in[0,1]$.
\end{itemize}
\end{lem}

By symmetry, if $\f_1(x) \leq \f_0(x)$, it follows that $\dot{\f}_0(x) \leq 0$. Moreover, if $\f_1(x)\leq \f_0(x)$ for all $x\in X$ then $\dot{\f}_t(x) \leq 0$ for a.e. $x,t $.
Here and in the sequel $\dot{\f_0}$, $\dot{\f_1}$ denote the right and left derivative, respectively while we recall that $\dot{\f}_t(x)$ is well defined for a.e $(x,t)$.

\begin{proof}
From Theorem \ref{thm:magicFormulaBdd} we know that
$d_p^p(\f_0,\f_1) =\int_X |\dot{\f}_{0}|^p MA(\f_0).$
Moreover, Proposition \ref{prop:Semmes} insures that $|\dot{\f}_{0}|\leq  ||\f_1-\f_0||_{L^{\infty}(X)}$. Hence, the first statement.

Assume $\dot{\f}_1(x) < 0$. Since $t \mapsto \f_t(x)$ is convex we infer
$\dot{\f_t}(x) \leq \dot{\f_1}(x)<0$. Thus $t \mapsto \f_t(x)$ is decreasing, hence $\f_1(x)< \f_0(x)$, a contradiction. This proves $(i)$.

Assume now that $\f_0(x)\leq \f_1(x)$ for all $x\in X$. Then
$$
\f_0\leq \f_t \leq \f_1.
$$
The first of the inequalities above follows from the fact that by Proposition \ref{prop:Semmes}
$$
\varphi= \sup\{u\quad u\in PSH(M, \omega)\;:\; u\leq \f_{0,1}\; {\rm on}\; M\}
$$
with $\f(x, t+is)= \f_t(x)$ and that $\f_0(x,t+is)= \f_0(x)$ is a subsolution 
(i.e. a candidate in the envelope). The other inequality follows from the fact that 
$\f_1(x,t+is)= \f_1(x)$ is a supersolution of (\ref{Semmes}) since 
$(\omega+dd^c_{x,z} \f_1)^{n+1}=0$ and $\f_1\geq \f_{0,1}$. The same argument shows that $\f_0\leq \f_s \leq \f_t$ for
all $0<s<t$ and $x \in X$,
hence $\dot{\f}_t(x) \geq 0$ for all $x\in X $ and a.e $t\in[0,1]$ since the derivative in time of $\f_t$ is well defined for a.e. $t$.
\end{proof}

We now establish a very useful relation established by Darvas \cite[Proposition 8.1]{Dar14}
when $\omega$ is K\"ahler (see also \cite[Corollary 4.14]{Dar15}).

\begin{prop} \label{prop:pyth}
Assume $\f_0,\f_1 \in  \H_{bd}$. Then for all $p \geq 1$,
$$
d_p^p(\f_0,\f_1)=d_p^p(\f_0,\f_0 \vee \f_1)+d_p^p(\f_0 \vee \f_1, \f_1).
$$
\end{prop}

\begin{proof}
 We proceed by approximation, so as to reduce to the K\"ahler case. The identity
 is known to hold for $d_{p,\e}$ and $\f_0 \vee_\e \f_1$, where
 $d_{p,\e}$ denotes the distance associated to the K\"ahler form $\omega_\e=\omega+\e \omega_X$
 and $\f_0 \vee_\e \f_1$ is the greatest $\omega_\e$-psh function that lies below $\min(\f_0,\f_1)$.
 
 Using Theorem \ref{thm:magicFormulaBdd} and  the triangle inequality,
 %and the symetry  between $\f_0$ and $\f_1$, 
 the proof boils down
 to check that $d_{p,\e}(\f_0 \vee \f_1, \f_0 \vee_\e \f_1) \rightarrow 0$ as $\e \rightarrow 0$. The same arguments used in the proof of Proposition \ref{quasi decreasing} yield 
 $$d_{p, \e}(\f_0 \vee \f_1, \f_0 \vee_\e \f_1) \leq d_{p, \e'}(\f_0 \vee \f_1, \f_0 \vee_\e \f_1), \quad \varepsilon< \varepsilon'.$$
%Set $\psi_0= \f_0 \vee \f_1$ and $\psi_1= \f_0 \vee_\e \f_1$ and fix $\varepsilon'\geq \varepsilon$. Let $\psi_t^\varepsilon$, $\psi_t^{\varepsilon'}$ denote the $\varepsilon$-geodesic and the $\varepsilon'$-geodesic both joining $\psi_0$ and $\psi_1$. Since $\varepsilon\rightarrow \psi_t^\varepsilon$ is increasing (Proposition \ref{cor:approx}) we have that for any $t\in(0,1)$
%$$\frac{\psi_t^\varepsilon-\psi_0}{t}\leq \frac{\psi_t^{\varepsilon'}-\psi_0}{t}$$
%that implies $\dot{\psi}^\e_0\leq \dot{\psi}^{\e'}_0$. Moreover observe that since $\psi_0(x)\leq \psi_1(x)$ for all $x\in X$, Lemma \ref{lem:maxprince} yields $\dot{\psi}^\e_0(x)\geq 0$ for all $x\in X$. It then follows that
%\begin{equation}\label{comparing_distaces_e}
%d_{p, \e}^p (\psi_0, \psi_1)=\int_X |\dot{\psi}^\e_0|^p \frac{(\omega_\varepsilon+dd^c \psi_0)^n}{V} \leq \int_X |\dot{\psi}^{\e'}_0|^p \frac{(\omega_{\varepsilon'}+dd^c \psi_0)^n}{V}=d_{p, \e'}^p(\psi_0, \psi_1). 
%\end{equation}
We claim that $d_{p, \e'}(\f_0 \vee \f_1, \f_0 \vee_\e \f_1)$ goes to zero as $\varepsilon$ goes to zero since $\f_0 \vee_\e \f_1$ decreases to $\f_0 \vee \f_1$ as $\e\rightarrow 0$. Indeed, observe that $\f_0 \vee \f_1, \f_0 \vee_\e \f_1\in \mathcal{E}^p(X,\omega_\e')\cap L^\infty(X)$ and by Proposition \ref{prop:upperbound} we know that 
$$d_{p, \e'}(\f_0 \vee \f_1, \f_0 \vee_\e \f_1)\leq 2 I_{p, \e'}(\f_0 \vee \f_1, \f_0 \vee_\e \f_1).$$
The same arguments in the proof of Proposition \ref{prop:approx2} then show that 
$I_{p, \e'}(\f_0 \vee \f_1, \f_0 \vee_\e \f_1)\rightarrow 0$ as $\varepsilon$ goes to zero. The conclusion then follows.
% $$
% d_{p,\e}(\f_0 \vee \f_1, \f_0 \vee_\e \f_1) \leq || \f_0 \vee \f_1-\f_0 \vee_\e \f_1||_{L^{\infty}(X)} \rightarrow 0,
% $$
% since $\f_0 \vee_\e \f_1$ decreases to $\f_0 \vee \f_1$ (Proposition \ref{cor:approx}).
\end{proof}

We note for later use the following consequence:

\begin{cor} \label{cor:min}
If  $\f_0,\f_1 \in  \H_{bd}$ then
$$
d_p(\f_0,\f_0 \vee \f_1) \leq d_p(\f_0,\f_1).
$$
\end{cor}

\subsection{Comparing $d_p$ and $I_p$}

The goal of this section is to establish that $d_p$ and $I_p$ are equivalent, extending 
\cite[Theorem 5.5]{Dar15}:

\begin{thm} \label{thm:equiv}
For all $\f_0,\f_1 \in  \H_{bd}$,
$$
2^{-1} d_p(\f_0,\f_1) \leq I_p(\f_0,\f_1) \leq 2^{4+(2n-1)/p}  d_p(\f_0,\f_1).
$$
\end{thm}

It follows from Definition \ref{distance_semipos} and Proposition \ref{prop:approx2} that
$$
d_p(\f_0,\f_1)=\lim_{\e \rightarrow 0} d_{p,\e}(\f_0,\f_1)
\text{ and }
I_p(\f_0,\f_1)=\lim_{\e \rightarrow 0} I_{p,\e}(\f_0,\f_1),
$$
so it suffices to establish these inequalities when $\omega$ is a K\"ahler form.

We  nevertheless give a direct proof, valid when $\omega$ is merely semi-positive, with several intermediate results  of independent interest.
Several of these results have been obtained by Darvas in \cite{Dar13,Dar14,Dar15}
when $\omega$ is K\"ahler.
%In all the proofs of this section we assume that $\omega$ is a K\"ahler form. As indicated above we can then pass to the limit, approximating $\omega$ by $\omega_\varepsilon=\omega+\varepsilon\omega_X$.

\begin{lem} \label{lem:dar2}
 Assume $\f_0,\f_1 \in  \H_{bd}$ satisfy $\f_0 \leq \f_1$. Then
 
 \smallskip
 
1) $d_p\left(\f_1, \frac{\f_0+\f_1}{2} \right)  \leq  d_p(\f_0, \f_1)$;

\smallskip

2) $d_p(\f_0,\f_1) \leq 2^{1+n/p} d_p(\f_0/2,\f_1/2)$;

\smallskip

3) if $\f_1=0$ then $d_p(\f_0,0) \geq 2 d_p(\f_0/2,0)$;

\smallskip

4) If $\p \in   \H_{bd}$ is such that $\f_0 \leq \p \leq \f_1$, then
$$
\max\{ d_p(\f_0,\p); d_p(\p,\f_1) \} \leq d_p(\f_0,\f_1).
$$
\end{lem}

\begin{proof}
%As already indicated we assume here, as well as in all the proofs of this section, that
%$\omega$ is a K\"ahler form. We can then pass to the limit, approximating
%$\omega$ by $\omega_\e=\omega+\e \omega_X$.

Let $\f_t$ (resp. $\p_t$) denote the Mabuchi geodesic joining $\f_0$ (resp.  $(\f_0+\f_1)/2$) to $\f_1$.
Since $\f_0 \leq \f_1 $, it follows from Lemma \ref{lem:maxprince}.ii that
$t \mapsto \f_t$, $t \mapsto \p_t$ are increasing and
$
{\f_t}\leq \p_t 
$
hence
$$
\frac{\f_t-\f_1}{t-1} \geq \frac{\p_t-\p_1}{t-1}
$$
since $\f_1=\p_1$.
Therefore
$
\dot{\f}_1 \geq \dot{\p}_1 \geq 0
$
and we infer
$$
 \int_X  |\dot{\p}_1|^p MA(\p_1)=d_p\left(\f_1, \frac{\f_0+\f_1}{2} \right)^p  \leq  d_p(\f_0, \f_1)^p
 =\int_X  |\dot{\f}_1|^p MA(\f_1).
$$
This together with Theorem \ref{thm:magicFormulaBdd} proves 1).

\smallskip

Let now $(\f_t)$ (resp. $(\p_t)$) denote the geodesic joining $\f_0$ to $\f_1$
(resp. $\f_0/2$ to $\f_1/2$). Observe that $t \mapsto \f_t,\p_t$ are increasing hence
$\dot{\f}_0 \geq 0$.
The family $(\f_t/2)$ is a subgeodesic joining $\f_0/2$ to $\f_1/2$, hence $\f_t/2 \leq \p_t$
and
$$
0 \leq \frac{\dot{\f}_0}{2} \leq \dot{\p}_0 \Longrightarrow 
\left| \dot{\f}_0 \right|^p \leq 2^p |\dot{\p}_0|^p.
$$
Moreover $MA(\f_0) \leq 2^n MA(\f_0/2)$, so we infer
$$
d_p(\f_0,\f_1)^p=\int_X \left| \dot{\f}_0 \right|^p MA(\f_0)
\leq 2^{n+p} d_p(\f_0/2,\f_1/2)^p,
$$
which proves 2).
A similar argument shows that 
$$
0 \leq \dot{\p}_1\leq \frac{\dot{\f}_1}{2} \Longrightarrow 
\left| \dot{\p}_1 \right|^p \leq 2^{-p} |\dot{\f}_1|^p.
$$
Now $MA(\f_1/2)=MA(\f_1)=MA(0)$ when $\f_1=0$, hence
$$
d_p(\f_0,0)^p=\int_X \left| \dot{\f}_1 \right|^p MA(0)
\geq 2^{p} d_p(\f_0/2,0)^p,
$$
which yields 3).

\smallskip

It remains to prove 4). Let $(\f_t)_{0 \leq t \leq 1}$ (resp. $(\p_t)_{0 \leq t \leq 1}$) be the geodesic joining $\f_0$ to $\f_1$
(resp. $\f_0$ to $\p$). Observe that $\f_0=\p_0$ and $\p_t \leq \f_t$, hence $\dot{\p}_0 \leq \dot{\f}_0$.
Moreover $0 \leq \dot{\p}_0$ since $t \mapsto \p_t(x)$ is increasing. We infer
$$
d_p(\f_0,\p)^p =\int_X |\dot{\p}_0|^p MA(\f_0) \leq \int_X |\dot{\f}_0|^p MA(\f_0)=d_p(\f_0,\f_1)^p.
$$
The other inequality is proved similarly.
\end{proof}

%\begin{lem} \label{lem:max}
%For all $\f_0,\f_1 \in \H$
%$$
%d_p(\f_0,\max(\f_0,\f_1)) \leq d_p(\f_0,\f_1).
%$$
%\end{lem}
%
%\begin{proof}
%We can assume without loss of generality that $\f_0=0$, 
%replacing $\omega$ by $\omega+dd^c \f_0$. Our goal is thus to show that
%$d_p(0,\max(0,\f_1)) \leq d_p(0,\f_1)$.
%
%We let $(\f_t)$ (resp. $(\p_t)$) denote the geodesic joining $0$ to $\f_1$
%(resp. to $\p_1=\max(0,\f_1)$). Observe that $\f_t \leq \p_t$ for all $t$.
%
%If $\f_1(x)>0$ then $\f_1(x)=\p_1(x)$ hence
%$
%0 \leq \dot{\p}_1(x) \leq \dot{\f}_1(x) \Rightarrow |\dot{\p}_1(x)|^p \leq |\dot{\f}_1(x)|^p.
%$
%If  $\f_1(x) \leq 0$, then $\p_0(x)=\p_1(x)=0$ hence $\p_t(x)=0$ for all $0 \leq t \leq 1$
%and $\dot{\p}_1(x)=0$. 
%%Thus $|\dot{\p}_1(x)|^p \leq |\dot{\f}_1(x)|^p$ in all cases
%It  follows therefore that
%\begin{eqnarray*}
%d_p(0,\max(0,\f_1))&=& \int_{\{0<\f_1 \}} |\dot{\p}_1|^p MA(\p_1)
%=\int_{\{0<\f_1 \}} |\dot{\p}_1|^p MA(\f_1) \\
%&\leq& \int_{\{0<\f_1 \}} |\dot{\f}_1|^p  MA(\f_1) \leq d_p(\f_0,\f_1).
%\end{eqnarray*}
%\end{proof}

\begin{prop} \label{prop:upperbound}
For all $\f_0,\f_1 \in  \H_{bd}$, 
$$
0 \leq d_p(\f_0,\f_1)  \leq 2 I_p(\f_0,\f_1).
$$
Moreover if $\f_0 \leq \f_1$ then $ I_p(\f_0,\f_1) \leq \left( \int_X (\f_1-\f_0)^p MA(\f_0) \right)^{1/p}$ and
$$
 d_p(\f_0,\f_1)  \leq  \left( \int_X (\f_1-\f_0)^p MA(\f_0) \right)^{1/p} \leq   2^{1+n/p} d_p(\f_0,\f_1).
$$
\end{prop}

\begin{proof} 
We first assume that $\f_0 \leq \f_1$. 
The inequality 
$$
 I_p(\f_0,\f_1) \leq \left( \int_X (\f_1-\f_0)^p MA(\f_0) \right)^{1/p}
$$ 
follows from Lemma \ref{lem:controleL2}.
Let 
$(\f_t)$ be the geodesic joining $\f_0$ to $\f_1$. It follows from Lemma \ref{lem:maxprince} that
$
0 \leq \dot{\f}_0 \leq \f_1-\f_0 \leq \dot{\f}_1
$
hence 
\begin{equation}\label{bound: f_1}
\int_X (\f_1-\f_0)^p MA(\f_1) \leq \int_X (\dot{\f}_1)^p MA(\f_1) =d_p(\f_0,\f_1)^p
\end{equation}
and similarly $d_p(\f_0,\f_1)^p \leq \int_X (\f_1-\f_0)^p MA(\f_0)$.

We now show that 
$ \int_X (\f_1-\f_0)^p MA(\f_0)   \leq   2^{n+p} d(\f_0,\f_1)^p$.
Observe that $\frac{\f_0+\f_1}{2} \in  \H_{bd}$ with 
$
 MA(\f_0)   \leq  2^n \, MA\left( \frac{\f_0+\f_1}{2} \right) 
$
hence 
\begin{eqnarray*}
\int_X (\f_1-\f_0)^p MA(\f_0) &=& 2^p \int_X \left(\frac{\f_0+\f_1}{2}-\f_0 \right)^p MA(\f_0) \\
& \leq & 2^{n+p}\int_X   \left(\frac{\f_0+\f_1}{2} -\f_0\right)^p  MA\left( \frac{\f_0+\f_1}{2} \right)  \\
& \leq & 2^{n+p} d_p\left(\f_0,\frac{\f_0+\f_1}{2} \right)^p,
\end{eqnarray*}
as follows from the first step of the proof since  $\f_0 \leq   \f_1 $.
Lemma \ref{lem:dar2}.4  yields
$$
d_p\left(\f_0,\frac{\f_0+\f_1}{2} \right) \leq d_p(\f_0,\f_1)
$$
hence 
$
\int_X (\f_1-\f_0)^p MA(\f_0) \leq    2^{n+p} d_p(\f_0,\f_1)^p.
$

\smallskip

We finally treat the first upper bound of the Proposition which does not require $\f_0$ to lie below $\f_1$. It follows from the triangle inequality that
{\small
\begin{eqnarray*}
\lefteqn{d_p(\f_0,\f_1) \leq  d_p(\f_0,\max(\f_0,\f_1))+d_p(\max(\f_0,\f_1), \f_1) }\\
&\leq & \left( \int_{\{\f_0<\f_1\}} (\f_1-\f_0)^p MA(\f_0) \right)^{1/p}+
\left( \int_{\{\f_0>\f_1\}} (\f_0-\f_1)^p MA(\f_1) \right)^{1/p} \\
&\leq& {2}^{1-1/p} \, \left( \int_X |\f_1-\f_0|^p \left[ MA(\f_0)+MA(\f_1)\right] \right)^{1/p} \\
&=& 2 \left( \int_X |\f_1-\f_0|^p \frac{ \left[ MA(\f_0)+MA(\f_1)\right]}{2} \right)^{1/p} 
\end{eqnarray*}
}
by using the elementary inequality ${a}^{1/p}+{b}^{1/p} \leq {2}^{1-1/p} {(a+b)}^{1/p}$. 
\end{proof}

\begin{rem}
Working with $\p=t \f_0+(1-t) \f_1$, $0<t<1$, instead of $\frac{\f_0+\f_1}{2}$, one can improve
the above inequality and obtain
$$
\left( \int_X (\f_1-\f_0)^p MA(\f_0) \right)^{1/p} \leq   \frac{(n+p)^{1+n/p}}{p \, n^{n/p}} d_p(\f_0,\f_1).
$$
\end{rem}

We now extend Lemma \ref{lem:dar2}.1, following \cite[Lemma 5.3]{Dar15}:

 \begin{lem} \label{lem:bary}
 For all $\f_0,\f_1 \in  \H_{bd}$,
 $$
 d_p\left(\f_0,\frac{\f_0+\f_1}{2} \right) \leq 2^{2+n/p} d_p(\f_0,\f_1).
 $$
 \end{lem}

 \begin{proof}
 When $\f_0 \leq \f_1$, this follows from Lemma \ref{lem:dar2}.1.
 Replacing $\omega$ by $\omega+dd^c \f_0$, we can assume without loss of generality that $\f_0=0$.
The triangle inequality yields
$$
 d_p\left(0,\frac{\f_1}{2} \right) \leq  d_p\left(0,0 \vee \frac{\f_1}{2} \right)+ d_p\left(0 \vee \frac{\f_1}{2},\frac{\f_1}{2} \right).
$$

 Observe that $0 \vee \f_1 \leq 0 \vee \frac{\f_1}{2} \leq \min(0,\frac{\f_1}{2}) $. It follows therefore from Lemma \ref{lem:dar2}.4 that
 $$
 d_p\left(0,0 \vee \frac{\f_1}{2} \right)+ d_p\left(0 \vee \frac{\f_1}{2},\frac{\f_1}{2} \right)
 \leq  d_p\left(0,0 \vee \f_1 \right)+ d_p\left(0 \vee \f_1,\frac{\f_1}{2} \right).
 $$
 
 Since $0 \vee \f_1 \leq 0$ and $0 \vee \f_1 \leq \frac{\f_1}{2}$, we can invoke Proposition \ref{prop:upperbound} to obtain

 \begin{eqnarray*}
 \lefteqn{d_p\left(0,0 \vee \f_1 \right)+ d_p\left(0 \vee \f_1,\frac{\f_1}{2} \right)} \\
 &&\leq  \left( \int_X | 0 \vee \f_1 |^p MA( 0 \vee \f_1) \right)^{1/p}
 + \left( \int_X | 0 \vee \f_1-  \frac{\f_1}{2} |^p MA( 0 \vee \f_1) \right)^{1/p} \\
 && \leq 2^{1-1/p}  \left( \int_X \left[ | 0 \vee \f_1 |^p +| 0 \vee \f_1-  \frac{\f_1}{2} |^p \right] MA( 0 \vee \f_1) \right)^{1/p}.
 \end{eqnarray*}

 Recall now that the measure $MA( 0 \vee \f_1)$ is supported on the contact set 
 $S:=\{x \in X \; ; \;  0 \vee \f_1(x)=\min(0,\f_1)(x) \}$. On this set we have
 $$
 | 0 \vee \f_1 |^p +| 0 \vee \f_1-  \frac{\f_1}{2} |^p
 \leq 2 |\f_1|^p=2   \left[ | 0 \vee \f_1 |^p +| 0 \vee \f_1- \f_1 |^p \right],
 $$
 while Proposition \ref{prop:upperbound}  yields
 \begin{eqnarray*}
\! \! \! \!  \! \! \! \!   \! \! \! \!   \! \! \! \!  \lefteqn{  \int_X \left[ | 0 \vee \f_1 |^p +| 0 \vee \f_1-  \f_1 |^p \right] MA( 0 \vee \f_1) } \\
 &&\leq  2^{p+n} \left[ d_p(0, 0 \vee \f_1)^p+d_p(0 \vee \f_1,\f_1)^p \right]  =2^{p+n} d_p(0,\f_1)^p,
 \end{eqnarray*}
where the last equality follows from Proposition \ref{prop:pyth}. Altogether this yields
$
 d_p\left(0,\frac{\f_1}{2} \right) \leq 2^{2+n/p} d_p(0,\f_1),
$
 as claimed.
 \end{proof}

We are now ready to prove Theorem \ref{thm:equiv}.
\begin{proof}
We have already observed  that $d_p(\f_0,\f_1) \leq 2 I_p(\f_0,\f_1)$ in Proposition \ref{prop:upperbound},
so we focus on the reverse control. Lemma \ref{lem:bary} and Proposition \ref{prop:pyth} yield 
\begin{eqnarray*}
2^{2p+n} d_p^p (\f_0, \f_1) &\geq & d_p^p \left(\f_0, \frac{\f_0+\f_1}{2} \right)\\
&=& d_p^p\left(\f_0,  \f_0\vee \frac{\f_0+\f_1}{2} \right) +  d_p^p\left(\frac{\f_0+\f_1}{2},  \f_0 \vee \frac{\f_0+\f_1}{2} \right)
\end{eqnarray*}
It follows from \eqref{bound: f_1} together with the fact that $2^n \MA\left(  \frac{\f_0+\f_1}{2}\right) \geq \MA(\f_0)$ that
$$ d_p^p\left(\f_0,  \f_0\vee \frac{\f_0+\f_1}{2} \right)  \geq   \int_X \left( \f_0- \frac{\f_0+\f_1}{2}\vee \f_0 \right)^p \MA(\f_0)$$
and

$$ d_p^p\left(\frac{\f_0+\f_1}{2},  \f_0\vee \frac{\f_0+\f_1}{2} \right) \geq 2^{-n} \int_X  \left(\frac{\f_0+\f_1}{2} - \f_0\vee \frac{\f_0+\f_1}{2}  \right)^p \MA(\f_0).$$
Hence

\begin{eqnarray*}
d_p^p (\f_0, \f_1) & \geq & 2^{-2(p+n)} \int_X \left[  \left( \f_0- \frac{\f_0+\f_1}{2}\vee \f_0 \right)^p+ \left( \frac{\f_0+\f_1}{2} -\frac{\f_0+\f_1}{2}\vee \f_0\right)^p \right] \MA(\f_0)\\
&\geq & 2^{1-3p-2n} \int_X  \left| \f_0- \frac{\f_0+\f_1}{2}\right|^p \MA(\f_0)\\
& =& 2^{1-4p-2n}\int_X |\f_0- \f_1|^p \MA(\f_0)
\end{eqnarray*}

where in the last inequality we used the fact that $|a-b|^p\leq 2^{p-1}( a^p+b^p)$, for any $a,b\in \R^+$.

Reversing the role of $\f_0$ adn $\f_1$ we get

$$d_p^p (\f_0, \f_1) \geq 2^{1-4p-2n} \int_X | \f_1- \f_0|^p \MA(\f_1)$$
from which it follows $d_p^p (\f_0, \f_1) \geq  2^{1-4p-2n}  I_p^p(\f_0, \f_1).$
\end{proof}

\subsection{Controlling the sup}

It follows from previous results that the supremum of a bounded potential with locally bounded laplacian in $\Amp(\alpha)$ is controlled by the distance
to the base point:

\begin{lem} \label{lem:controlesup}
There exists $C>0$ such that for all $\f \in  \H_{bd}$,
$$
-2^{4+2n}d_1(0,\f) \leq \sup_X \f \leq 2^{4+2n} (n+1) d_1(0,\f)+C
$$
\end{lem}

\begin{proof}
If $\sup_X \f \leq 0$, then $\sup_X \f \leq 0 \leq (n+1) d_1(0,\f)+C$, while
$$
-d_1(0,\f)=E(\f) \leq \sup_X \f,
$$
as follows from Proposition \ref{lem:formuled1}.
We therefore assume in the sequel that $\sup_X \f \geq 0$.
If $\f \geq 0$, then Proposition \ref{lem:formuled1} yields
$$
\frac{1}{n+1} \int_X \f MA(0) \leq E(\f) =d_1(0,\f).
$$
 It is a classical consequence of the $\omega$-plurisubharmonicity \cite[Proposition 2.7]{GZ05} that there exists $C>0$ such that
 such that for all $\f \in PSH(X,\omega)$,
$$
\sup_X \f \leq \int_X \f \, MA(0)+C.
$$
Thus $\sup_X \f \leq (n+1) d_1(0,\f)+C$.

  When $\sup_X \f \geq 0$ but $\f$ takes both positive and negative values,
  we set $\p=\max(0,\f)$ and observe that $\sup_X \p=\sup_X \f$.
Using  Propositions \ref{prop:python}, \ref{prop:upperbound} and Theorem \ref{thm:equiv} we obtain

 $$
 d_1(0,\max(0,\f)) \leq 2 I_1(0,\max(0,\f) ) \leq 2 I_1(0,\f )\leq 2^{5-(2n-1)/p}d_1(0,\f).
 $$ 
  The conclusion follows therefore from the previous case.
\end{proof}

\begin{prop} \label{lem:formuled1}
Assume $\f,\p \in  \H_{bd}$. Then
$$
d_1(\f,\p)=E(\f)+E(\p)-2 E(\f \vee \p).
$$
\end{prop}

\begin{proof}
We proceed by approximation, so as to reduce to the K\"ahler case. By \cite[Corollary 4.14]{Dar15} we know that $$d_{1,\e}(\f,\p)=E_{\omega_\e}(\f)+E_{\omega_e}(\p)-2 E_{\omega_\e}(\f \vee_{\e} \p)$$
 where $\omega_\e:= \omega+ \e \omega_X$, $\f \vee_{\e} \p$ is the greatest $\omega_\e$-psh function that lies below $\min(\f, \p)$ and $E_{\omega_\e}$ is as in \eqref{energy_approx}.   Since $(\omega_\e+dd^c \f)^n$ converges weakly to $(\omega+dd^c \f)^n$ we have that $E_{\omega_\e} (\f)$ converges to $E (\f)$ as $\e$ goes to $0$. The same holds for $E_{\omega_\e}(\psi)$. We then need to insure that $E_{\omega_\e}(\f \vee_{\omega_\e} \psi)$ converges to $E(\f \vee \psi)$. Denote $\phi_\e:= \f \vee_{\omega_\e} \psi$ and $\phi:=\f \vee \psi $. Fix $\e' > \e$. Using Lemma \ref{lem:AubinMabuchi1} and the fact that $\phi_\e$ is decreasing to $\phi$ we get
\begin{eqnarray*}
0\geq E_{\omega_\e}(\phi_\e) - E_{\omega_\e}(\phi)&=& \frac{1}{(n+1)V_\e} \sum_{j=0}^{n} \int_X (\phi_\e-\phi) (\omega_\e+dd^c \phi_\e)^{j} \wedge (\omega_\e+dd^c \phi)^{n-j}\\
&\geq & \frac{1}{(n+1)V_\e} \sum_{j=0}^{n}\int_X (\phi_{\e'}-\phi) (\omega+\omega_X+dd^c \phi_\e)^{j} \wedge (\omega+\omega_X+dd^c \phi)^{n-j}.
\end{eqnarray*}
Letting first $\e$ to zero and then $\e'$ we get the result. The conclusion then follows from the arguments above and Proposition \ref{quasi decreasing} .
\end{proof}

\section{The complete geodesic space $({\mathcal E}^p(X,\omega),d_p)$} \label{sec:geod}

%We recall here and in the follwing that $H$ denotes the space $\HH_{new}$.

\subsection{Metric completion}

For $\f,\p \in \E^p(X,\omega)$ we let $\f_j,\p_k$ denote sequences of elements in $\H_{bd}$
decreasing to $\f,\p$ respectively, and set
$$
D_p(\f,\p):=\liminf_{j,k \rightarrow +\infty} d_p(\f_j,\p_k).
$$

%When $\omega=\pi^* \omega_Y$, the existence of these smooth approximants follows from \cite{EGZ15}.
%For the general case, it is actually necessary to work with $\omega_\e$-psh approximants and let $\e$ decreases to $0^+$.

We list in the proposition below various properties of this extension.

\begin{prop}\label{properties $D_p$}
\text{ }

i) $D_p$ is a distance on $\E^p(X,\omega)$
which coincides with $d_p$ on $\H_{bd}$;

ii) the definition of $D_p$ is independent of the choice of the approximants;

iii) $D_p$ is continuous along decreasing sequences in $\E^p(X,\omega)$.

\smallskip Moreover all previous inequalities comparing  $d_p$ and $I_p$ on $\H_{bd}$
extend to inequalities between $D_p$ and $I_p$ on $\E^p(X,\omega)$.
\end{prop}

In the sequel we will therefore denote $D_p$ by $d_p$.

\begin{proof}
It is a tedious exercise to verify that $D_p$ defines a "semi-distance", i.e. satisfies all properties of a distance
but for the separation property. It follows from the definition of $D_p$ and Proposition \ref{prop:approx2} that Theorem \ref{thm:equiv} extends in a natural way to potentials in $\E^p(X, \omega)$.  
%$$ 2^{-1} d_p(\f, \p)\leq I_p(\f, \p) \leq 2^{1-1/p} d_p(\f, \p).$$
If $D_p(\f,\p)=0$, it follows therefore that
$I_p(\f,\p)=0$ hence $\f=\p$ by the domination principle.

One can check that $D_p$ coincides with $d_p$ on $\H_{bd}$ as follows:
%It follows from Lemma \ref{lem:maxprince} that
%$$
%d_p(\f,\f_j) \leq \| \f-\f_j \|_{L^{\infty}(X)} \rightarrow 0
%$$
%and similarly $d_p(\p,\p_k) \rightarrow 0$ hence the triangle inequality yields $D_p(\f,\p)=d_p(\f,\p)$.
using ii) one can use the constant sequences $\f_j \equiv \f$ and $\p_k \equiv \p$
to obtain this equality.

We now prove ii). Let $\f_j, u_j$ (resp. $\p_k, v_k$) denote two sequences of
elements of $\H_{bd}$ decreasing to $\f$ (resp. $\p$). We can assume without loss of generality
that these sequences are intertwining,
i.e. for all $j,k \in \N$, there exists $\ell,q \in \N$ such that
$
\f_j \leq u_{\ell}
\; \text{ and }  \; 
\p_k \leq v_{q},
$
with similar reverse inequalities. It follows from Proposition \ref{prop:upperbound}
and the triangle inequality that
\begin{eqnarray*}
\left| d_p(\f_j,\p_k)-d_p(u_\ell,v_q) \right| &\leq& 
d_p(\f_j,u_\ell)+d_p(\p_k,v_q) \\
&\leq&  2 I_p(\f_j,u_\ell)+2 I_p(\p_k,v_q).
\end{eqnarray*}
Now, again by Proposition \ref{prop:upperbound} we get
$$
I_p(\f_j,u_\ell)^p \leq    \int_X (u_\ell-\f_j)^p MA(\f_j)  \leq   (p+1)^n \int_X (u_\ell-\f)^p MA(\f)
$$
where the last inequality follows from \cite[Lemma 3.5]{GZ07}. 
The monotone convergence theorem therefore yields  $I_p(\f_j,u_\ell)+I_p(\p_k,v_q) \rightarrow 0$ 
as $\ell, q \rightarrow +\infty$,
proving ii).

One shows iii) with similar arguments. The extension of the inequalities comparing $d_p$ and 
$I_p$  follows from \cite[Theorem 2.17]{BEGZ10}.
\end{proof}

 \begin{prop}\label{thm_Completeness}
The metric spaces $(\E_{norm}^p(X,\omega),d_p)$ and  $(\E^p(X,\omega),d_p)$ are complete. 
The Mabuchi topology $d_p$ dominates the
topology induced by $I$: if  a sequence converges for $d_p$, then it converges in energy.
\end{prop}

\begin{proof}
Let $(\f_j) \in \E^p(X,\omega)^{\N}$ be a Cauchy sequence for $d_p$. We claim that there exists $\p \in \E^p(X,\omega)$ such that
$$
d_p(\f_j,\p) \rightarrow 0
\text{ and }
I(\p,\f_j) \rightarrow 0.
$$

Extracting and relabelling, we can assume that 
$$
d_p(\f_j,\f_{j+1}) \leq 2^{-j}, \quad j\geq 1.
$$
Set $\f_{-1} \equiv 0$ and for $k \geq j$,
$
\p_{j,k}:=\f_j \vee \f_{j+1} \vee \cdots \vee \f_k,
$
and observe that $\p_{j,k}:=\f_j \vee \p_{j,k+1}$. Hence the Pythagorean formula gives 
$$
d_p( \f_j, \p_{j,k}) \leq d_p( \f_j,  \p_{j+1,k})
\leq 2^{-j} +d_p( \f_{j+1}, \p_{j+1,k}).
$$
Repeating this argument we get 
$
d_p( \f_j, \p_{j,k}) \leq 2^{-j+1}.
$
 We then have
\begin{eqnarray*}
d_p(0,\p_{j,k}) &\leq& \sum_{\ell=-1}^{j-1} d_p(\f_\ell,\f_{\ell+1}) +d_p( \f_j, \p_{j,k}) \\
&\leq & \sum_{\ell=-1}^{j} d_p(\f_\ell,\f_{\ell+1}) +d_p( \f_{j+1}, \p_{j+1,k})\\
&\leq &d_p(0,\f_1)+ 2+ 2^{-j+1}.
\end{eqnarray*}

It follows from Theorem \ref{thm:equiv} that 
$I_p(0,\p_{j,k})$ is uniformly bounded, hence its decreasing limit
 $\p_j:=\lim_{k \rightarrow +\infty} \p_{j,k} \in \E^p(X,\omega)$ \cite[Proposition 2.19]{BEGZ10}.
From above we also have
$$ d_p(0,\p_{j})  \leq d_p(0,\f_1)+ 2+ 2^{-j+1}.$$

Lemma \ref{lem:controlesup} then ensures that  $(\sup_X \p_j)_j$ is uniformly bounded, hence $\p_j$ increases a.e. towards $\psi\in \psh(X, \omega)$. Also, $\p \in \E^p(X,\omega)$ thanks to \cite[Proposition 2.4]{BEGZ10}. Moreover,
\cite[Theorem 2.17]{BEGZ10} 
yields
$$
I(\p,\p_j)+I_p(\p_j,\p) \longrightarrow  0.
$$

It follows therefore from Proposition \ref{prop:upperbound} that
$d_p(\p,\p_j) \rightarrow 0$ and
$$
d_p(\p,\f_j) \leq d_p(\p,\p_j)+d_p(\p_j,\f_j) \leq d_p(\p,\p_j)+2^{-j+1} \rightarrow 0.
$$
Recalling that $\p_j \leq \f_j$, it follows from the quasi-triangle inequality, Proposition \ref{prop: I and I_1} and Theorem \ref{thm:equiv} that
$$
I(\p,\f_j) \leq c_n \left\{ I(\p,\p_j)+I(\p_j,\f_j) \right\}
\leq c_{n,p} \left\{ I(\p,\p_j)+d_p(\p_j,\f_j)  \right\}\rightarrow 0.
$$
\end{proof}

Recall that the {\it precompletion} of a metric space $(X,d)$
is the set of all Cauchy sequences $C_X$ of $X$, together with the semi-distance
$$
\d(\{x_j\},\{y_j\})=\lim_{j \rightarrow +\infty} d(x_j,y_j).
$$
The metric {\it completion} $(\overline{X},d)$ of $(X,d)$ is the quotient space $C_X/\sim$, where
$$
\{x_j\} \sim \{y_j\} \Longleftrightarrow \d(\{x_j\},\{y_j\})=0,
$$
equipped with the induced distance that we still denote by $d$.

%Recall that a {\it path metric space} is a metric space for which the distance between any two points coincides with the infimum of the lengths of rectifiable curves joining the two points. By construction
%the space $(\H,d)$ is a path metric space. For such metric spaces, an alternative description of the metric completion can be obtained as follows: consider $C_X'$ the set of all rectifiable curves 
%$\g:(0,1] \rightarrow X$ equipped with the semi-distance
%$$
%\d(\g,\tilde{\g}):=\lim_{t \rightarrow 0} d(\g(t),\tilde{\gamma}(t)).
%$$
%The metric completion $(\overline{X},d)$ is then the quotient space $C_X'/\sim$ which identifies
%zero-distance curves $\g,\tilde{\g}$.
%
%\smallskip
%
%Both constructions yield a rather abstract view on the metric completion. 
We are now taking advantage of the fact that $\H_{bd}$ lives inside the complete metric space $(\E^p(\a),d_p)$ to conclude that:

\begin{thm}\label{completion}
The metric completion of $({\H}_{bd},d_p)$  is isometric to
$(\E^p(X,\omega),d_p)$.
\end{thm}

Thanks to Theorem \ref{thm:equiv}, an equivalent formulation of the above statement is that the metric completion of $({\H}_{bd},d_p)$ is bi-Lipschitz equi\-valent to $(\E^p(X,\omega),I_p)$.

\begin{proof}
We work at the level of normalized potentials, 
$$
 \E^p_{0}(X,\omega)=\{ \f \in \E^p(X,\omega) \, | \, E( \f)=0 \}
$$
and $\H_{0}:=\{ \f \in \HH_{bd} \, | \, \omega+dd^c \f \geq 0 \text{ and } E( \f)=0 \}$.

Since $(\E^p_{0}(X,\omega),d_p)$ is a complete metric space that contains 
$\H_{0}$, it suffices to show that the latter is dense in $\E^p_{0}(X,\omega)$.
Fix $\f \in \E^p_{0}(X,\omega)$ and let $(\f_j) \in {\H}_{0}^\N$ be a
sequence quasi-decreasing to $\f$ :
the normalization condition $E(\f_j)=0$ prevents from getting a truly decreasing sequence, 
however $\f_j+\e_j$ is decreasing where $\e_j$ is  a sequence of real numbers decreasing to zero.
It follows from Proposition \ref{prop:upperbound} that
$$
d_p(\f_{j+\ell}+\varepsilon_{j+l},\f_j+{\varepsilon_j})^p \leq \int_X (\f_j-\f_{j+\ell})^p MA(\f_{j+\ell})+\e_j.
$$
Now \cite[Lemma 3.5]{GZ07} shows that the latter is bounded from above by
$$
(p+1)^n \int_X (\f_j-\f)^p MA(\f)+\e_j
$$
which converges to zero as $j \rightarrow +\infty$, as follows from the monotone convergence theorem.
Therefore $(\f_j)$ is a Cauchy sequence in $({\H}_{0} ,d_p)$ which converges to $\f$ since 
$$
0 \leq d_p(\f,\f_j+\varepsilon_j) \leq 
\liminf_{\ell \rightarrow +\infty} d_p(\f_{j+\ell},\f_j) 
\leq 2 (1+p)^{n/p} I_p(\f_j,\f) +\e_j^{1/p} \rightarrow 0
$$ 
by Proposition \ref{prop:upperbound} and \cite[Theorem 2.17]{BEGZ10}.

\smallskip

We note the following alternative approach of independent interest.
One  first shows that ${\H}_{0}$ is dense in the set of all bounded $\omega$-psh functions.
Given $\f \in \E_{0}^p(X,\omega)$ one then considers its ``canonical approximants" 
$$
\f_j=\max(\f,-j) +\e_j \in PSH_{0}(X,\omega) \cap L^{\infty}(X)
$$
which  decrease towards $\f \in \E_0^p(X,\omega)$.
It follows from Proposition \ref{prop:upperbound} that 
\begin{eqnarray*}
\lefteqn{ d_p(\f_{j+\ell},\f_j)^p \leq  o(1)+\int_X (\f_j-\f_{j+\ell})^p MA(\f_{j+\ell})}   \\
&&=o(1)+\int_{(\f \leq -j-\ell)} \ell^p MA(\f_{j+\ell}) +\int_{(-j-\ell<\f<-j)} (\f_j-\f_{j+\ell})^p MA(\f)  \\
&&=o(1)+\int_{(\f \leq -j-\ell)} \ell^p MA(\f)+\int_{(-j-\ell<\f<-j)} (\f_j-\f_{j+\ell})^p MA(\f)  \\
&&  \leq o(1)+\int_{(\f<-j)} \f^p MA(\f),
\end{eqnarray*}
where we have used the maximum principle together with the fact that 
$$
\int_{(\f \leq -k)}  MA(\f_{k})=\int_X MA(\f_{k})-\int_{(\f>-k)} MA(\f_{k})=
\int_{(\f \leq -k)} MA(\f),
$$
since $\f \in \E(X,\omega)$,
as follows again from the maximum principle.
We infer that $(\f_j)$ is a Cauchy sequence which converges to $\f$.
\end{proof}

 We are now in position to prove Theorem B of the introduction:
 
\begin{cor}\label{thm: completion}
Assume $\omega=\pi^*\omega_Y$, where $\omega_Y$ is a Hodge form. Then the metric completion $(\overline{\H}_\a,d_p)$  is isometric to
$(\E^p(\a),d_p)$. Similarly the metric completion $(\overline{\H}_\omega,d_p)$  is isometric to
$(\E^p(X,\omega),d_p)$.
\end{cor}

\begin{proof}
Thanks to \cite[Corollary C]{CGZ} we can insure that the space $\H_\omega$ is dense in $\H_{bd}$. The result then follows from Theorem \ref{completion}.
\end{proof}

\subsection{Weak geodesics}

%\begin{rem}\label{Bound:L_infty}
% Let $(\f_j)$ be a sequence of functions in $\HH$ uniformly converging to 
% $\f \in \HH$. Let $\f_t$ denote the geodesic joining $\f_j$ to $\f$. 
%It follows from Proposition \ref{prop:Bo} that  for all $0 \leq t \leq 1$,
%$$
%d_p(\f,\f_j) =\left(\int_X |\dot{\f_t}|^p MA(\f_t) \right)^{1/p} \leq ||\f-\f_j||_{L^{\infty}(X)}.
%$$
%We will see in Example \ref{exa:explicit} that the convergence in the Mabuchi sense is much weaker than the uniform convergence.
%\end{rem}

\subsubsection{Finite energy geodesics}

We now define finite energy geodesics joining two finite energy endpoints $\f_0,\f_1 \in \E^1(X,\omega)$.
Fix $j \in \N$ and consider
$\f_0^{j}, \f_1^{j}$ bounded $\omega$-psh functions decreasing to $\f_0,\f_1$.
We let $\f_{t,j}$ denote the bounded geodesic joining $\f_0^{j}$ to $\f_1^{j}$.
 It follows from the maximum principle
that $j \mapsto \f_{t,j}$ is non-increasing. 
We can thus set
$$
\f_t:=\lim_{j \rightarrow +\infty} \f_{t,j}.
$$

\begin{defi}
The map $(t,x) \mapsto \f_t(x)$ is the (finite energy) Mabuchi geodesic joining $\f_0$ to $\f_1$.
\end{defi}

The $\f_t$'s form indeed a family of finite energy functions : since 
$t\mapsto E(\f_{t,j})$ is affine (Lemma \ref{lem:AubinMabuchi1}), we infer  for all $j \in \N$,
$$
 (1-t) E(\f_0) +t E(\f_1) \leq  (1-t) E(\f_0^{(j)}) +t E(\f_1^{(j)}) =E(\f_{t,j}),
$$
hence $\f_t \in \E^1(X,\omega)$ with $ (1-t) E(\f_0) +t E(\f_1) =E(\f_t)$.

It follows from the maximum principle that $\f_t$ is independent of the choice of the approximants
$\f_0^{j}, \f_1^{j}$:
if we set  $\varphi(x,z):=\f_t(x)$, $z=t+is$, then $\varphi$ is a maximal
$\omega$-psh function in $X \times S$, as a decreasing limit of maximal $\omega$-psh functions.
It is thus the unique maximal $\omega$-psh function in $X \times S$ with boundary values
$\f_0,\f_1$.

\smallskip
  
When $\f_0,\f_1$ belong to $\E^p(X,\omega)$, these weak geodesics are again {\it metric geodesics}
in the complete metric space $(\E^p(X,\omega),d_p)$:

\begin{prop}\label{prop:gedisically closed}
Given $\f_0,\f_1 \in \E^p(X,\omega)$, the Mabuchi geodesic $\f$ joining 
$\f_0$ to $\f_1$ lies in $\E^p(X,\omega)$ and  satisfies,
 for all $t,s \in [0,1]$,
$$
d_p(\f_t,\f_s)=|t-s| \, d_p(\f_0,\f_1).
$$

Thus $(\E^p(X,\omega),d_p)$ is a geodesic space.
\end{prop}

\begin{proof}
We can assume without loss of generality that $\f_0,\f_1 \leq 0$.
Fix $j \in \N$ and consider
$\f_0^{j}, \f_1^{j}$ bounded $\omega$-psh functions decreasing to $\f_0,\f_1$.
We let $\f_{t,j}$ denote the bounded geodesic joining $\f_0^{j}$ to $\f_1^{j}$,
which decreases towards $\f_t$ as $j$ increases to $+\infty$.
Observe that
$$
\f_0 \vee \f_1 \leq \f_0^{j} \vee \f_1^{j} \leq \f_{t,j}.
$$
It follows therefore from \cite[Lemma 3.5]{GZ07} and Lemma \ref{lem:epgeod} that
$$
\int_X (-\f_{t,j})^p MA(\f_{t,j}) \leq 
(p+1)^n  \int_X  (-\f_0 \vee \f_1)^p MA(\f_0 \vee \f_1) <+\infty
$$
hence the monotone convergence theorem yields 
$
\int_X (-\f_{t})^p MA(\f_{t}) <+\infty,
$
for all $t$,
i.e. $\f_t \in \E^p(X,\omega)$.

The remaining assertion is proved as in the case of bounded geodesics (Proposition \ref{prop:bddmin}).
\end{proof}

\begin{lem} \label{lem:epgeod}
Assume $0 \geq \f_0,\f_1 \in \E^p(X,\omega)$. Then 
$\f_0 \vee \f_1 \in \E^p(X,\omega)$ and
$$
 \int_X (-\f_0 \vee \f_1)^p MA(\f_0 \vee \f_1) \leq  \int_X (-\f_0)^p MA(\f_0)+ \int_X (-\f_1)^p MA( \f_1).
$$
\end{lem}

\begin{proof}
It suffices to establish the claimed inequality when $\f_0,\f_1 \in \HH_{bd}$ and then proceed by approximation.
It follows from Proposition \ref{pro:kis} that
$$
MA(\f_0 \vee \f_1) \leq \1_{\{ \f_0 \vee \f_1=\f_0\}} MA(\f_0)+\1_{\{ \f_0 \vee \f_1=\f_1\}} MA(\f_1).
$$
The inequality follows since $\f_0,\f_1 \leq 0$.
\end{proof}

\subsubsection{(Non) uniqueness of geodesics}

Fix $\f_0,\f_1 \in \E^1(X,\omega)$. If the sets $(\f_0<\f_1)$ and $(\f_0>\f_1)$ are both non empty,
the function $\f_0 \vee \f_1$ differs from $\f_0$ and $\f_1$ and it follows from Proposition \ref{prop:pyth}
that 
$$
d_1(\f_0,\f_1)=d_1(\f_0, \f_0 \vee \f_1) +d_1(\f_0 \vee \f_1,\f_1),
$$
thus the concatenation of the  geodesic joining $\f_0$ to $\f_0 \vee \f_1$ and that joining $\f_0 \vee \f_1$ to $\f_1$
gives another minimizing path joining $\f_0$ to $\f_1$.

When $\f_0 \leq \f_1$, this argument does not work anymore, but there are nevertheless very many minimizing paths,
as shown by the following result:

\begin{lem}
Assume $\f_0,\f_1 \in \HH_{bd}$ are such that $\f_0 \leq \f_1$. Let $(\p_t)_{0 \leq t \leq 1}$ be a path 
joining $\f_0$ to $\f_1$. Then
$$
\ell_1(\p)=d_1(\f_0,\f_1) \Longleftrightarrow \dot{\p}_t(x) \geq 0, \; {\rm for\; a.e.}\, t,x.
$$

In particular $t \mapsto t\f_1(x)+(1-t)\f_0$ is a minimizing path for $d_1$ which is not a Mabuchi geodesic, unless
$\f_1-\f_0$ is constant.
\end{lem}

\begin{proof}
Observe that 
\begin{eqnarray*}
\ell_1(\p)&=& \int_0^1 \int_X \left| \dot{\p}_t(x) \right| MA(\p_t) \, dt
 \geq  \left| \int_0^1 \int_X  \dot{\p}_t(x)  MA(\p_t) \, dt \right| \\
 &=& \left| \int_0^1 \frac{d}{dt}E(\p_t) \, dt  \right|
 =\left| E(\f_1)-E(\f_0)  \right| =d_1(\f_0,\f_1)
 \end{eqnarray*}
 where the last identity follows from Proposition \ref{lem:formuled1}. There is equality iff $|\dot{\p}_t(x)|=\dot{\p}_t(x) \geq 0$ for a.e. $(t,x)$ (the sign has to be positive because $\p_0=\f_0 \leq \f_1=\p_1$).
 
 In particular $t \mapsto \p_t=t\f_1(x)+(1-t)\f_0$ has this property, since $\dot{\p}_t=\f_1-\f_0 \geq 0$.
We recall that, since $\psi_t$ is a smooth path, the geodesic equation can be written as
$$\ddot{\p}_t\MA(\p_t)= \frac{n}{V}d\dot{\psi_t}\wedge d^c \dot{\psi_t}  \wedge \omega_{\p_t}^{n-1}$$
(see Section \ref{subsec: Mabuchi}). Now $\ddot{\p}_t =0$ hence $t \mapsto \p_t$ is not a Mabuchi geodesic, unless
 $
 d( \f_1-\f_0)\wedge d^c (\f_1-\f_0) \wedge \omega_{\p_t}^{n-1} =0
 $
 for all $t$, i.e. $\f_1-\f_0$ is contant.
\end{proof}

On the other hand it follows from the work of Darvas \cite[Lemma 6.12]{Dar14} (based on \cite[Section 2.4]{CC02}) that geodesics are unique in $\E^2(X,\omega)$:

\begin{thm} Assume $\omega=\pi^* \omega_Y$, where $\omega_Y$ is a Hodge form. Then the space $(\E^2(X,\omega),d_2)$ is a CAT(0) space.
\end{thm}

Complete CAT(0) spaces are also called Hadamard spaces. Recall that a CAT(0) space is a geodesic space which has non positive curvature in the sense of Alexandrov. Hadamard spaces enjoy many interesting properties (uniqueness of geodesics, contractibility, convexity properties,...see \cite{BH99}).

\begin{proof}
By Corollary \ref{thm: completion} we know that $(\E^2(X,\omega),d_2)$ is the completion of $(\H_\omega, d_2)$ and by Proposition \ref{prop:gedisically closed} it is a geodesic metric space. \cite[Exercise 1.9.1.c (p. 163)]{BH99} insures that $(\E^2(X,\omega),d_2)$ is a  CAT(0) space if and only if the CN inequality of Bruhat-Tits \cite{BT72} holds, i.e. $\forall P,Q,R\in \E^2(X,\omega)$ and for any $M\in \E^2(X,\omega)$ such that $d_2(Q,M)=d_2(R, M)=d_2(Q,R)/2$ (in other words $M=\varphi^{QR}_t|_{t=1/2}$ where $\varphi_t^{QR}$ is the geodesic joining $Q, R$) one has
\begin{equation}\label{CN}
d_2(P, M)^2 \leq \frac{1}{2} d_2(P,Q)^2 + \frac{1}{2} d_2(P,R)^2-\frac{1}{4} d_2(Q,R)^2.
\end{equation}

Assume first that $P,Q, R\in \H_\omega \subset \H_{\omega_\e}$. Then by \cite[Section 2.4]{CC02} (see also \cite[Lemma 6.12]{Dar14}) we have that

$$d_{2, \e}(P,M_\e)^2 \leq \frac{1}{2} d_{2, \e}(P,Q)^2 + \frac{1}{2} d_{2, \e}(P,R)^2-\frac{1}{4} d_{2, \e}(Q,R)^2$$
where $M_\e$ is the point of $\e$-geodesic joining $Q,R$ such that $d_{2, \e}(Q,M)=d_{2,\e}(R, M)=d_{2, \e} (Q,R)/2$. Thanks to Theorem \ref{thm:approx0} the RHS in the inequality above converges to the RHS of \eqref{CN} as $\e$ goes to zero. We claim that $d_{2, \e}(P,M_\e)$ converges to $d(P, M)$. Observe first that $M_\e$ decreases to $M$ since $\e$-geodesics decreases as $\e$ decrease to zero (Proposition \ref{cor:approx}). Moreover, the triangle inequality yields $|d_{2, \e}(P,M_\e)- d_{2, \e}(P,M)| \leq d_{2, \e}(M,M_\e)$. Since $M, M_\e$ are both bounded, it follows from Theorem \ref{thm:equiv} and Proposition \ref{prop:approx2} that $d_{2, \e'}(M,M_\e)\rightarrow 0$ as $\e\rightarrow 0$. This proves the claim.\\
If $P,Q,R\in \E^2(X, \omega)$ we choose smooth approximants $P_k, Q_k, R_k \in \H_\omega$ decreasing to $P, Q, R$. The above arguments insure that
\begin{equation}\label{CN approx}
d_2(P_k, M_k)^2 \leq \frac{1}{2} d_2(P_k,Q_k)^2 + \frac{1}{2} d_2(P_k,R_k)^2-\frac{1}{4} d_2(Q_k,R_k)^2.
\end{equation}
The comparison principle implies that $M_k$ decreases to $M$ as $k$ goes to $+\infty$. It then follows from Proposition \ref{prop:upperbound} and Proposition \ref{properties $D_p$} that $d_2(M, M_k)\rightarrow 0$ as $k$ goes to $+\infty$. This together with Proposition \ref{properties $D_p$} gives \eqref{CN} when letting $k\rightarrow +\infty$. 
%Calabi and Chen proved in \cite[Theorem 1.1]{CC02} that $(\H_\omega, d_2)$
%satisfies the CN inequality (\ref{CN}) in the case when the reference form $\omega$ is K\"ahler. 
%%Approximating $\omega$ by $\omega_\e=\omega+\e \omega_X$ 
%The result extends to our present setting
%by approximation (Theorem \ref{thm:approx0}).
%
%Moreover, the CN inequality extends to $\E^2(X,\omega)$ by density.
%It follows therefore from
%\cite[Corollary II.3.11]{BH99} that 
% $(\overline{\H}_\omega,d_2)$ is a CAT(0) space.
\end{proof}

%\section{Constant scalar curvature K\"ahler metrics} \label{sec:csck}

\section{Singular K\"ahler-Einstein metrics of positive curvature} \label{sec:csck}

The existence of singular K\"ahler-Einstein metrics of non-positive curvature has been established in \cite{EGZ09}, generalizing the
fundamental work of Aubin \cite{Aub78} and Yau \cite{Yau78}. They always exist, provided
the underlying variety has mild singularities and the first Chern class is non-positive.

Singular K\"ahler-Einstein metrics of positive curvature are more difficult
to construct. It is already so in the smooth case \cite{CDS3}. 
Their first properties have been obtained in \cite{BBGZ,BBEGZ}.
In Section \ref{sec:critanal}, pushing further these works, we provide a necessary and sufficient analytic  condition for their existence, generalizing
a result of Tian \cite{Tian97} and Phong-Song-Sturm-Weinkove \cite{PSSW08}.

 \subsection{Log terminal singularities}\label{sec:klt}
A \emph{pair} $(Y,D)$ is the data of a connected normal compact complex variety $Y$ and an effective $\Q$-divisor $D$ such that $K_Y+D$ is $\Q$-Cartier. 
We write 
$$
Y_0:=Y_{{\rm{reg}}}\setminus{\rm{Supp }} D.
$$ 
Given a log resolution $\pi:X\to Y$ of $(Y,D)$ (which may  be chosen to be an isomorphism over $Y_0$), 
there exists a unique $\Q$-divisor $\sum_i a_i E_i$ whose push-forward to $Y$ is $-D$ and such that 
$$
K_{X}=\pi^*(K_Y+D)+\sum_i a_i E_i. 
$$
%The coefficient $a_i\in\Q$ is known as the \emph{discrepancy} of $(X,D)$ along $E_j$.

\begin{defi}
The pair $(Y,D)$ is \emph{klt}  if $a_j>-1$ for all $j$. 
\end{defi}

The same condition will then hold for all log resolutions of $Y$. When $D=0$, one 
says that $Y$ is \emph{log terminal} when the pair $(Y,0)$ is klt.
%The discrepancies $a_i$ admit 
We have the following analytic interpretation. Fix $r \in \N^*$  such that $r(K_Y+D)$ is Cartier. 
If $\sigma$ is a nowhere vanishing section of the corresponding line bundle over a small open set $U$ of $Y$ then 
\begin{equation}\label{equ:adaptedloc}
\left(i^{r n^2}\sigma\wedge\bar\sigma\right)^{1/r}
\end{equation}
defines a smooth, positive volume form on $U_0:=U\cap Y_0$. If $f_j$ is a local equation of $E_j$ around a point of $\pi^{-1}(U)$, 
then
$$
\pi^*\left(i^{r n^2}\sigma\wedge\bar\sigma\right)^{1/r}=\prod_i|f_i|^{2a_i}dV
$$
locally on $\pi^{-1}(U)$ for some local volume form $dV$. Since $\sum_i E_i$ has normal crossings, 
this shows that $(Y,D)$ is klt iff each volume form of the form (\ref{equ:adaptedloc}) has locally finite mass near singular points of $Y$.

The previous construction globalizes as follows: 

\begin{defi}\label{defi:adapted} 
Let $(Y,D)$ be a pair and let $\phi$ be a smooth Hermitian metric on the $\Q$-line bundle $-(K_Y+D)$. 
The corresponding \emph{adapted measure} ${\rm mes}_\phi$ on $Y_{\rm{reg}}$ is locally defined by choosing 
a nowhere zero section $\sigma$ of $r(K_Y+D)$ over a small open set $U$ and setting
$$
{\rm mes}_\phi:=\left(i^{r n^2}\sigma\wedge\overline{\sigma}\right)^{1/r}/|\sigma|_{r\phi}^{2/r}.
$$
\end{defi}

The point is that the measure ${\rm mes}_\phi$ does not depend on the choice of $\sigma$, hence is globally defined. 
The above discussion shows that 
$$
(Y,D)
\text{ is klt }
\Longleftrightarrow
{\rm mes}_\phi
\text{ has finite total mass on } Y,
$$
 in which case we view it as a Radon measure on the whole of $Y$.

\subsection{K\"ahler-Einstein metrics on log Fano pairs}\label{sec:KE}

\begin{defi}\label{defi:logfano} 
A \emph{log Fano pair} is a klt pair $(Y,D)$ such that $Y$ is projective and $-(K_Y+D)$ is ample. 
\end{defi}

Let $(Y,D)$ be a log Fano pair. Fix a reference smooth strictly psh metric $\phi_0$ on $-(K_Y+D)$, with curvature $\om_0$ and adapted measure $\mu_0={\rm mes}_{\phi_0}$. 
We normalize $\phi_0$ so that $\mu_0$ is a probability measure. The volume of $(Y,D)$ is 
$$
V:=c_1(Y,D)^n=\int_X\om_0^n.
$$

\begin{defi}\label{defi:KE} 
A \emph{K\"ahler-Einstein metric} $T$ for the log Fano pair $(Y,D)$ is a finite energy current
$T \in   c_1(Y,D)$ such that
$
  T^n=V \cdot \mu_T.
$
\end{defi}

We now list some important properties of these objects established in \cite{BBGZ,Bern13,BBEGZ}:
 \begin{itemize}
\item  A K\"ahler-Einstein metric $\om$ is automatically smooth on $Y_0$, with continuous potentials on $Y$, and it satisfies
$$
\Ric(\om_{KE})=\om_{KE}+[D]
\text{ on } Y_{\rm{reg}}.
$$
\item  The definition of a log Fano pair requires the singularities to be klt. This condition is in fact necessary to obtain K-E metrics on $Y_{\rm{reg}}$.
\item  The K\"ahler-Einstein equation reads
$
(\om_0+dd^c\phi)^n=e^{-\phi+c}\mu_0
$
for some constant $c\in\R$. If we choose a log resolution, the equation becomes 
$({\om}+dd^c \f)^n=e^{-\f+c} \widetilde{\mu}_0$, where $\om=\pi^*\om_0$ is semipositive and big and $\widetilde{\mu}_0=\prod_i|f_i|^{2a_i}dV$. 
\item The potential  $\f$ belongs to $\H_\omega$ and maximizes the functional
$$
{\mathcal F}(\f) :=E(\f)+\log \left[ \int_{\tilde{X}} e^{-\f} d\widetilde{\mu}_0 \right].
$$
Conversely any maximizer of ${\mathcal F}$ is a K\"ahler-Einstein metric.
\item Two K\"ahler-Einstein metrics are connected by the flow of a holomorphic vector field that leaves $D$ invariant.
\item If the functional ${\mathcal F}$ is {\it proper} 
(i.e. if $E(\f_j) \rightarrow -\infty \Rightarrow {\mathcal F}(\f_j) \rightarrow -\infty$), 
then there exists a unique K\"ahler-Einstein metric.
 \end{itemize}

Here $[D]$ is the integration current on $D|_{Y_{\rm{reg}}}$. Writing $\Ric(\om_{KE})$ on $Y_{\rm{reg}}$ implicitely means that the positive measure $\om_{KE}^n|_{Y_{\rm{reg}}}$ corresponds to a singular metric on $-K_{Y_{\rm{reg}}}$, whose curvature is then $\Ric(\om_{KE})$ by definition.

\subsection{The analytic criterion} \label{sec:critanal}

Following and idea of Darvas-Rubinstein \cite{DR15}, 
we now extend \cite[Theorem 1.6]{Tian97} and \cite{PSSW08}
by proving the following:
%\marginpar{Here $\f$ is the $\omega$-psh function on the smooth resolution X. Should I say that more explicitely? } NO :)
\begin{thm}
Let $(Y,D)$ be a log Fano pair. It admits a unique K\"ahler-Einstein metric iff 
there exists $\e,M>0$ such that for all $\f \in \H_{norm}$,
$$
{\mathcal F}(\f) \leq -\e d_1(0,\f)+M.
$$
\end{thm}

\noindent This is Theorem D of the introduction.

\begin{proof}
We are going to use Theorem B. Note that $\omega_Y\in c_1(-K_X-D)$ is a Hodge form. One implication is due to \cite[Theorems 4.8 and  5.4]{BBEGZ}: if 
$$
{\mathcal F}(\f) \leq -\e d_1(0,\f)+M, 
$$
then ${\mathcal F}$ is proper, hence there exists a unique K\"ahler-Einstein metric.

So we assume now that there exists $\omega$ a unique K\"ahler-Einstein metric, which we take
as our base point of $\HH_\omega$. It is the unique maximizer of ${\mathcal F}$ on $\E^1(X,\omega)$,
$$
{\mathcal F}(0)=\sup_{\f \in \E^1(X,\omega)} {\mathcal F}(\f),
$$
as follows from \cite[Theorem 6.6]{BBGZ}, \cite[Theorems 4.8 and 5.3]{BBEGZ}.

Note that ${\mathcal F}$ is invariant by translations, so we actually consider the restriction of 
${\mathcal F}$ on $\E_{norm}^1(X,\omega)=\{ \f \in \E^1(X,\omega), \; \sup_X \f=0\}$.
Assume for contradiction that there is no $\e>0$ such that 
${\mathcal F}(\f) \leq -\e d_1(0,\f)+M$ for all $\f \in \H_{norm}$,
where we set $M:={\mathcal F}(0)+1$.
Then we can find a sequence $(\f_j) \in \H_\omega^\N$ such that 
$\sup_X \f_j=0$ and 
$$
{\mathcal F}(\f_j) > - \frac{d_1(0,\f_j)}{j+1}+{\mathcal F}(0)+1.
$$

If $E(\f_j)$ does not blow up to $-\infty$, we  reach a contradiction:
up to extracting and relabelling, we can assume that $E(\f_j)$ is bounded
and $\f_j$ converges to some $\p \in \E^1(X,\omega)$.
Since ${\mathcal F}$ is upper semi-continuous, we infer
${\mathcal F}(\p) \geq {\mathcal F}(0)+1$, a contradiction.
 
So we assume now that $E(\f_j) \rightarrow -\infty$.
It follows from Lemma \ref{lem:formuled1} that $d_j:=d_1(0,\f_j)=-E(\f_j) \rightarrow +\infty$.
We let $(\f_{t,j})_{0 \leq t \leq d_j}$ denote the Mabuchi geodesic with unit speed joining $0$ to $\f_j$ 
and set $\p_j:=\f_{1,j}$. Note that the arguments in Lemma \ref{lem:maxprince} show that $t \mapsto \f_{t,j}$ is decreasing, hence
$
\f_j \leq \p_j \leq 0.
$
In particular $\sup_X \p_j=0$, while by definition $d_1(0,\p_j)=1=-E(\p_j)$.

It follows now from Berndtsson's convexity result \cite[Section 6.2]{Bern13}
and its generalization to the singular context \cite[Theorem 11.1]{BBEGZ}
that the map $t \mapsto {\mathcal F}(\f_{t,j})$ is concave. We infer
$$
0 \geq {\mathcal F}(\f_{1,j})-{\mathcal F}(\f_{0,j}) \geq \frac{{\mathcal F}(\f_{d_j,j})-{\mathcal F}(\f_{0,j})}{d_j}
>-\frac{1}{j+1}+\frac{1}{d_j},
$$
thus ${\mathcal F}(\p_j) \rightarrow {\mathcal F}(0)$.
This shows that $(\p_j)$ is a maximizing sequence for ${\mathcal F}$ which therefore strongly converges to $0$,
by \cite[Theorem 5.3.3]{BBEGZ}. This yields a contradiction since $d_1(0,\p_j)=1$.
\end{proof}

\begin{ackn}
We thank  T.Darvas, H.C.Lu and A.Zeriahi for useful conversations.
The first named author is supported by a Marie Sklodowska Curie individual fellowship 660940–KRF–CY (MSCA–IF). 
This article is based upon work supported also by the NSF Grant  DMS-1440140,
 while the first author was in residence at the MSRI,  during the Spring 2016 semester.
It is partially based on lecture notes of the second author \cite{G14}, after
series of lectures he gave at KIAS in April 2013. 
The authors thank J.-M.Hwang and M.Paun for their invitation and the staff of KIAS and MSRI
for providing excellent conditions of work.
\end{ackn}

\end{document}